\documentclass[10pt,a4paper]{article}

\usepackage{a4wide}
\usepackage[english]{babel}

\usepackage{graphicx}
\usepackage{subfig}

\usepackage{amsmath}
\usepackage{amsthm}
\usepackage{amssymb}

\usepackage{color}      % Colored text, used for \TODO macro
\usepackage{xspace} % Useful in custom macros to ensure spacing after it
\usepackage{enumitem}  % Improved enumerate style, used for axioms
\newcommand{\url}[1]{\texttt{#1}}

\newcommand{\F}{\ensuremath{\mathbb{F}}}    % Finite field
\newcommand{\abs}[1]{\lvert#1\rvert}        % abs: |foo|

\newcommand{\gen}[1]{\langle{#1}\rangle}

\newcommand{\eps}{\varepsilon}

\newcommand{\isomorph}{\cong}

\DeclareMathOperator{\Aut}{Aut}
\DeclareMathOperator{\Fix}{Fix}
\DeclareMathOperator{\id}{id}
\DeclareMathOperator{\Id}{id}
\DeclareMathOperator{\Inv}{Inv}
\DeclareMathOperator{\charac}{char}

\DeclareMathOperator{\Sym}{Sym}
\DeclareMathOperator{\Stab}{Stab}
\DeclareMathOperator{\proj}{proj}

\newcommand{\mc}[1]{\mathcal{#1}}

\newcommand{\mcC}{\mathcal{C}}

\newcommand{\mouf}{\mathbb{M}} % Moufang sets

\newcommand{\op}{\; {\rm opp} \; }
\newcommand{\Uals}{\{U_\alpha\}_{\alpha\in\Phi}}

\newcommand{\thcod}{$\theta$-codistance\xspace}

\newcommand{\Ct}{\mathcal{C}^\theta}
\newcommand{\Gt}{G_\theta}
\newcommand{\Rt}{R^\theta}
\newcommand{\Qt}{Q^\theta}
\newcommand{\dt}{\delta^\theta}
\newcommand{\lt}{l^\theta}	% TODO: Change this to \ell??
\newcommand{\wtheta}{\tilde{\theta}}

\newenvironment{axioms}[1]{\begin{enumerate}[label={\sffamily\bfseries({#1}\arabic*)}, ref={({#1}\arabic*)}, leftmargin=*]}{\end{enumerate}}

\theoremstyle{plain}
  \newtheorem{theorem}{Theorem}[section]
  
  \newtheorem{lemma}[theorem]{Lemma}
  \newtheorem{lem}[theorem]{Lemma}
  \newtheorem{proposition}[theorem]{Proposition}
  \newtheorem{prop}[theorem]{Proposition}
  
  \newtheorem{cor}[theorem]{Corollary}

  \newtheorem{thm*}{Theorem}

\theoremstyle{definition}
  \newtheorem{defnPLAIN}[theorem]{Definition}
  \newtheorem{example}[theorem]{Example}

  \newtheorem*{defn*}{Definition}

  \newtheorem{remark}[theorem]{Remark}

\newenvironment{defn}{\pushQED{\qed}\begin{defnPLAIN}
}{\popQED\end{defnPLAIN}}

\makeatletter
\renewcommand\section{\@startsection {section}{1}{\z@}%
                                   {-3.5ex \@plus -1ex \@minus -.2ex}%
                                   {2.3ex \@plus.2ex}%
                                   {\boldmath\normalfont\Large\bfseries}}
\renewcommand\subsection{\@startsection{subsection}{2}{\z@}%
                                     {-3.25ex\@plus -1ex \@minus -.2ex}%
                                     {1.5ex \@plus .2ex}%
                                     {\boldmath\normalfont\large\bfseries}}
\renewcommand\subsubsection{\@startsection{subsubsection}{3}{\z@}%
                                     {-3.25ex\@plus -1ex \@minus -.2ex}%
                                     {1.5ex \@plus .2ex}%
                                     {\boldmath\normalfont\normalsize\bfseries}}
\renewcommand\paragraph{\@startsection{paragraph}{4}{\z@}%
                                    {3.25ex \@plus1ex \@minus.2ex}%
                                    {-1em}%
                                    {\boldmath\normalfont\normalsize\bfseries}}
\renewcommand\subparagraph{\@startsection{subparagraph}{5}{\parindent}%
                                       {3.25ex \@plus1ex \@minus .2ex}%
                                       {-1em}%
                                      {\boldmath\normalfont\normalsize\bfseries}}

\makeatother

\newcommand\Defn[1]{{\boldmath\bfseries#1}}

\begin{document}

\title{Abstract involutions of algebraic groups \\ and of Kac--Moody groups}

\author{Ralf Gramlich\thanks{The first author gratefully acknowledges a Heisenberg fellowship by the Deutsche Forschungsgemeinschaft.} \\ Max Horn \\ Bernhard M\"uhlherr}

\maketitle

\begin{abstract}
Based on the second author's thesis \cite{Horn:2008b} in this article we provide a uniform treatment of abstract involutions of
algebraic groups and of Kac--Moody groups using twin buildings, RGD systems,
and twisted involutions of Coxeter groups. Notably we simultaneously 
generalize the double coset decompositions established in
\cite{Springer:1984}, \cite{Helminck/Wang:1993} for algebraic groups and in
\cite{Kac/Wang:1992} for certain Kac--Moody groups, we analyze the
filtration studied in \cite{Devillers/Muehlherr:2007} in the context of arbitrary
involutions, and we answer a structural question on the combinatorics of
involutions of twin buildings raised in 
\cite{Bennett/Gramlich/Hoffman/Shpectorov:2003}. 
\end{abstract}

% -*- mode: latex; TeX-master: "diss"; -*-
%!TEX root = diss.tex
% $Id: flips.tex 140 2008-12-17 22:44:06Z mhorn $

\section{Introduction}

For the study and the classification of reductive symmetric $k$-varieties, i.e., homogeneous spaces $G_k/H_k$, where $G$ is a reductive algebraic group defined over $k$ and $H$ an open subgroup of the group of fixed points of a $k$-involution $\theta$ of $G$, it is crucial to have a precise description of involutions and their fine structure of restricted root systems with multiplicities and Weyl groups.
A full classification of symmetric $k$-varieties for arbitrary $k$ depends on the understanding of three invariants (cf.\ \cite[Section 1]{Helminck:2000}): the classification of Satake diagrams, of isomorphism classes of involutions of anisotropic groups, and of isomorphism classes of quadratic elements. 

Over the reals this classification is well known (\cite{Berger:1957}, \cite[Chapter X]{Helgason:1978}, \cite{Helminck:1988}). The key feature of the real numbers and, more generally, the real closed fields  that makes symmetric varieties over these fields classifiable is that they have a finite codegree within the field of complex numbers. Since by the Artin--Schreier theorem (\cite[Theorem 11.14]{Jacobson:1989}) admitting an algebraically closed field as an extension field of finite degree is a characteristic feature of real closed fields, it is therefore not surprising that many difficult problems arise when studying rationality properties of symmetric $k$-varieties for non-real closed fields $k$, as documented in \cite{Helminck/Wang:1993}.
One main focus of \cite{Helminck/Wang:1993} are symmetric varieties over non-archimedean local fields, which makes the theory of Bruhat--Tits buildings (\cite{Bruhat/Tits:1972}, \cite{Bruhat/Tits:1984}) applicable in this context, as illustrated in \cite{Benoist/Oh:2007} and \cite{Delorme/Secherre}.

\bigskip
The purpose of the present paper is to use the theory of groups with an RGD system (see \cite[Section 3.3]{Tits:1992} and also \cite[Definition 7.82 and Section 8.6.1]{Abramenko/Brown:2008}), which generalizes the theories of connected isotropic reductive linear algebraic groups (Remark \ref{remalg}) and of split Kac--Moody groups (Remark \ref{remkm}), in order to study involutions of groups with algebraic origin uniformly. It turns out that many properties of $k$-involutions of connected reductive $k$-groups (established in \cite{Helminck/Wang:1993} over arbitrary infinite fields $k$ of characteristic distinct from two) and of algebraic involutions of split Kac--Moody groups (established in \cite{Kac/Wang:1992} over arbitrary algebraically closed fields of characteristic zero) are in fact true for arbitrary abstract involutions of arbitrary groups with an RGD system of ``characteristic'' distinct from two. (Here we use the word characteristic in the sense of the geometric and algebraic properties studied in Section \ref{sec:$2$-divisible-rootgrps}; see Remarks \ref{algchar2} and \ref{prop:mouf-2-divis}.) The abstract approach we describe in this article is possible because of the existing classifications of the abstract automorphisms of connected isotropic reductive linear algebraic groups (\cite[Proposition 7.2]{Borel/Tits:1973}) and of split Kac--Moody groups (\cite[Theorem 4.1]{Caprace:2009}); cf.\ Remark \ref{abstractpossible} below.

\bigskip
One of our main results is a double coset decomposition that simultaneously generalizes \cite[Proposition~6.10]{Helminck/Wang:1993} and \cite[Proposition~5.15]{Kac/Wang:1992}. (We refer to Sections \ref{basics} and \ref{sec:flip-intro} for definitions.)

\bigskip \noindent {\bf Theorem \ref{cor:DCD-A-twodivis}.}
{\em Let $(G,\Uals,T)$ be an RGD system, let $\mcC$ be the associated twin
  building, let $\theta$ be a quasi-flip of $G$, let $B$ be a Borel subgroup
  of $G$, let $\Gt$ be the subgroup of $G$ of $\theta$-fixed elements, and let
  $\{ \Sigma_i \mid i \in I\}$ be a set of representatives of the
  $\Gt$-conjugacy classes of $\theta$-stable twin apartments of $\mcC$. If the
  root groups are uniquely $2$-divisible or if $G$ is split algebraic/Kac--Moody and $\theta$ is a semi-linear flip, then
\[\Gt \backslash G / B \isomorph \bigcup_{i \in I} W_{\Gt}(\Sigma_i) \backslash W_{G}(\Sigma_i) , \] 
where $W_{X}(\Sigma_i) := \Stab_X(\Sigma_i) / \Fix_X(\Sigma_i)$.}

\bigskip
The image under the above $1$-$1$ correspondence of a double coset $\Gt g B$ is determined by the unique $\Gt$-conjugacy class of $\theta$-stable twin apartments containing the chamber $gB$ represented by one of the $\Sigma_i$ (cf.\ Proposition \ref{thm:theta-stable-apt} and Lemma \ref{lem:intersecting-theta-stable-apts-Gt-conj}), and the $\Stab_{\Gt}(\Sigma_i)$-orbit structure of $\Sigma_i$. This generalizes the observation $\Gt \backslash G / B \isomorph \Inv(W)$ made in \cite{Gramlich/Muehlherr:unpub-lattices} in the case of a split Kac--Moody group $G$ defined over a finite field $\F_{q^2}$ and an $\F_{q^2}$-semi-linar flip $\theta$ of $G$; cf.\ Corollary \ref{cor}. In this case the $1$-$1$ correspondence is given by $\Gt gB \mapsto \dt(gB)$ (see Definition \ref{theta-cod}).

\bigskip
The situation when $|\Gt \backslash G / B| = 1$, i.e., when $G$ admits an Iwasawa decomposition $G = \Gt B$, has been studied in detail in \cite{Medts/Gramlich/Horn}. The main result of that paper is a characterization of the fields $\F$ such that some split algebraic group over $\F$ or $\F$-locally split Kac--Moody group admits an Iwasawa decomposition.

\bigskip
Sections \ref{sec:strong-flips} and \ref{sec:flip-flop-residual-conn} constitute the building-theoretic heart of the present article. We adapt the filtration studied in \cite{Devillers/Muehlherr:2007} to the situation of non-strong quasi-flips (cf.\ Definition \ref{non-strong}) and provide a local-to-global result on the structure of flip-flop systems (cf.\ Definition \ref{dfn-flipflop}). A combination of our findings with a result by Hendrik Van Maldeghem and the second author (see Theorem \ref{hvm})  yields our second main result, which answers a question posed in \cite{Bennett/Gramlich/Hoffman/Shpectorov:2003} in the most relevant cases. 

\bigskip \noindent {\bf Theorem \ref{thm:quite-many-flips-are-geometric}.}
{\em Let 
\begin{itemize}
\item $(G,\Uals,T)$ be an RGD system of $2$-spherical type with finite root groups $\Uals$ of odd order and of cardinality at least five, or 
\item let $\mathbb{F}$ be an infinite field of characteristic distinct from two and let $G$ be a connected $\F$-split reductive
$\F$-group or a split Kac--Moody
group over $\F$ of $2$-spherical type $(W,S)$.
\end{itemize}
 Moreover, let $\mcC$ be the associated twin building, and let $\theta$ be a quasi-flip of $G$.

 Then the flip-flop system $\Ct$ is connected and equals the union of the minimal Phan residues of $\mcC_\eps$. Moreover, there exists $K \subseteq S$, such that $(\mcC, \theta)$ is $K$-homogeneous. Furthermore, the $K$-residue chamber system $\Ct_K$ is connected and residually connected. In particular, if $\theta$ is proper, then $(\mcC, \theta)$ is $\emptyset$-homogeneous, and $\Ct$ is residually connected.
}

\bigskip
Standard methods from geometric group theory allow us to conclude the following finite generation result from the preceding theorem.

\bigskip \noindent {\bf Theorem \ref{thm:loc-fin-KM-is-fin-gen}.}
{\em   Let 
$(G,\Uals,T)$ be an RGD system of $2$-spherical type with finite root groups $\Uals$ of odd order and of cardinality at least five, and let $\theta$ be a quasi-flip of $G$.
Then the group $\Gt$ is finitely generated.}

\bigskip
The proofs of Theorems \ref{thm:quite-many-flips-are-geometric} and \ref{thm:loc-fin-KM-is-fin-gen} are based on abstract building-theoretic arguments and a local analysis of the rank two case. This local analysis has been conducted by Hendrik Van Maldeghem and the second author for all finite and for all split Moufang polygons with uniquely $2$-divisible root groups (see Theorem \ref{hvm} on page \pageref{hvm}). We do not see any reason why this local analysis should fail for arbitrary Moufang polygons with sufficiently large uniquely $2$-divisible root groups, but currently there is no formal proof available.

\subsubsection*{Acknowledgement}

The authors thank Aloysius Helminck for several very encouraging discussions on the topic of this article. The authors also thank Tonny Springer and Wilberd van der Kallen for pointing out to them the work by Aloysius Helminck in the first place.

%%%%%%%%%%%%%%%%%%%%%%%%%%%%%%%%%%%%%%%%%%%%%%%%%%%%%%%%%%%%%%%%%%%%%%%%%%%%%
%%%%%%%%%%%%%%%%%%%%%%%%%%%%%%%%%%%%%%%%%%%%%%%%%%%%%%%%%%%%%%%%%%%%%%%%%%%%%

\section{Basics} \label{basics}

\subsubsection*{RGD systems}

Let $G$ be a group endowed with a family $\Uals$ of subgroups, indexed by a root system $\Phi$ of type $(W,S)$, and let $T$ be another subgroup of $G$.
Following \cite[Definition~7.82 and Section~8.6.1]{Abramenko/Brown:2008} (also \cite[Section~3.3]{Tits:1992} and \cite[Section~2]{Caprace/Remy:2008-LN}) the triple $(G, \Uals, T)$ is called an {RGD system of type $(W,S)$} if it satisfies the following, where $U_\pm:=\gen{U_\alpha\mid \alpha\in\Phi_\pm}$:
\begin{axioms}{RGD}
\setcounter{enumi}{-1}   % Start counting at 0
\item For all $\alpha\in\Phi$, we have $U_\alpha\neq\{1\}$.
\item For every prenilpotent pair $\{\alpha,\beta\}\subset\Phi$ of distinct roots, we have $[U_\alpha,U_\beta]\subset\gen{U_\gamma\mid \gamma \in ]\alpha,\beta[}$.
\item For every $s\in S$ and $u\in U_{\alpha_s}\setminus\{1\}$, there exist elements $u',u''$ of $U_{-\alpha_s}$ such that the product $\mu(u):=u'uu''$ conjugates $U_\beta$ onto $U_{s(\beta)}$ for each $\beta\in\Phi$.
\item For all $s\in S$ we have $U_{-\alpha_s} \not\subseteq U_+$.
\item $G=T.\gen{U_\alpha\mid\alpha\in\Phi}$.
\item $T$ normalizes $U_\alpha$ for each $\alpha\in\Phi$, i.e., 
  $T \leq \bigcap_{\alpha\in\Phi} N_G(U_\alpha)$. 
\end{axioms}
The groups $U_\alpha$ are called the \Defn{root subgroups} of $G$ and the pair $(\Uals,T)$ is referred to as a \Defn{root group datum}.

\medskip
Define $X_\alpha:=\gen{U_\alpha, U_{-\alpha}}$ and $X_{\alpha,\beta}:=\gen{X_\alpha,X_\beta}$. A root group datum is called \Defn{locally split}, if the group $T$ is abelian and if, for each $\alpha \in \Phi$, there is a field $\F_\alpha$ such that $X_\alpha$ is isomorphic to $\mathrm{(P)SL}_2(\F_\alpha)$ and $\{ U_\alpha, U_{-\alpha} \}$ is isomorphic to its natural root group datum of rank one.
A locally split root group datum is called {\bf $\F$-locally split}, if $\F_\alpha = \F$ for all $\alpha \in \Phi$.

\begin{remark} \label{remalg}
Let $G$ be a connected reductive linear algebraic group defined over an infinite field $\F$ and let $G(\F)$ denote the subgroup of $\F$-rational points of $G$; cf.\ \cite{Humphreys:1975}, \cite{Borel:1991}, \cite{Springer:1998}.
Assume that $G$ is isotropic over $\F$ and let $T$ be a maximal $\F$-split $\F$-torus. By \cite{Borel/Tits:1965} there exists a family of root groups $\Uals$, indexed by the relative root system $\Phi$ of $(G(\F),T(\F))$, such that $(G(\F), \Uals, T(\F))$ is an RGD system.
For details see also \cite[Section 6]{Bruhat/Tits:1972}, \cite[Section~7.9]{Abramenko/Brown:2008}. 
\end{remark}

\begin{example} \label{An}
Let $n \geq 2$ and let $(W,S)$ be the Coxeter system of type $A_{n-1}$, i.e., let $W \cong S_n$ be the symmetric group on $n$ letters and let $S = \{ s_1, ..., s_{n-1} \}$ be its standard generating set consisting of adjacent transpositions. The root system $\Phi$ of type $(W,S)$ has one root $\alpha_{i,j}$ for each ordered pair $(i,j)$ with $1 \leq i, j \leq n$ and $i \neq j$; cf.\ \cite[Examples 1.119, 3.52]{Abramenko/Brown:2008}, \cite[Section~2.10]{Humphreys:1990}.
Let 
\begin{itemize}
\item $\mathbb{F}$ be a field, 
\item $G = \mathrm{GL}_n(\mathbb{F})$, considered as the group of invertible $(n \times n)$-matrices over $\mathbb{F}$,
\item $T := \{ \mathrm{diag}(\lambda_1,...,\lambda_n) \mid \lambda_i \in \mathbb{F}^* \}$ be the group of invertible diagonal $(n \times n)$-matrices over $\mathbb{F}$,
\item $U_{\alpha_{i,j}} := \{ 1_{n \times n} + \lambda e_{ij} \mid \lambda \in \mathbb{F} \}$, where $1_{n \times n}$ denotes the identity $(n \times n)$-matrix and $e_{ij}$ the matrix whose $(i,j)$-entry is $1$ and all other entries are $0$.
\end{itemize}
Then the triple $(G,\{ U_{\alpha_{i,j}} \}_{\alpha_{i,j} \in \Phi}, T)$ is an $\mathbb{F}$-locally split RGD system of type $(W,S)$; cf.\ \cite[Example 7.133]{Abramenko/Brown:2008}. 
Replacing $G$ by $G' = \mathrm{SL}_n(\mathbb{F})$ and $T$ by $T' := \{ \mathrm{diag}(\lambda_1,...,\lambda_{n-1},\lambda_1^{-1}\cdots\lambda_{n-1}^{-1}) \mid \lambda_i \in \mathbb{F}^* \}$ also yields an $\mathbb{F}$-locally split RGD system $(G',\{ U_{\alpha_{i,j}} \}_{\alpha_{i,j} \in \Phi}, T')$ of type $(W,S)$. 
\end{example}

\begin{remark} \label{remkm}
Let $\mathcal{D} = (I,A,\Lambda,(c_i)_{i\in I})$ be a Kac--Moody root datum, let $\mathcal{F} = (\mathcal{G},(\phi_i)_{i \in I},\eta)$ be the basis of a Tits functor $\mathcal{G}$ of type $\mathcal{D}$, let $\mathbb{F}$ be a field, and let $G := \mathcal{G}(\mathbb{F})$ be the corresponding split Kac--Moody group; see \cite{Tits:1987} and also \cite{Tits:1992}, \cite[Part II]{Remy:2002}, \cite{Caprace:2009}. 
Let $M(A) = (m_{ij})_{i,j \in I}$ be the Coxeter matrix defined by $m_{ii} := 1$ and, for $i \neq j$, by $m_{ij} := 2, 3, 4, 6$ or by $m := \infty$ according to whether the product $A_{ij}A_{ji}$ equals $0$, $1$, $2$, $3$ or is greater or equal $4$. Let $(W,S)$ be a Coxeter system of type $M(A)$ with $S = \{ s_i \mid i \in I\}$, let $\Phi$ be its root system, and let $\Pi = \{ \alpha_i \mid i \in I \}$ be a system of fundamental roots of $\Phi$ such that, for each $i \in I$, the reflection associated with $\alpha_i$ is $s_i$. For each $i \in I$ let $U_{\alpha_i}$ (resp.\ $U_{-\alpha_i}$) be the image under $\phi_i(\F)$ of the subgroup of upper (resp.\ lower) triangular unipotent matrices.
Then, via conjugation of the $U_{\pm \alpha_i}$ with elements of $W$, there exists a family of root groups $\{ U_\alpha \}_{\alpha\in \Phi^{\mathrm{re}}}$ such that $(G(\F), \{ U_\alpha \}_{\alpha\in \Phi^{\mathrm{re}}}, T)$ is an RGD system, where $T := \bigcap_{\alpha\in \Phi^{\mathrm{re}}} N_{G(\F)}(U_\alpha)$; cf.\ \cite{Tits:1987}, \cite{Tits:1992}, \cite[Proposition 8.4.1]{Remy:2002}, \cite[Lemma 1.4]{Caprace:2009}.

Almost split Kac--Moody groups satisfying the Galois descent property defined in \cite[12.1.1]{Remy:2002} also admit a root group datum by \cite[Theorem~12.6.3]{Remy:2002}. See \cite{Muehlherr:1999} for additional examples of Kac--Moody groups with a root group datum.
\end{remark}

\begin{example}\label{Antilde}
Let $n \geq 2$, let $(W,S)$ be the Coxeter system of type $A_{n-1}$, and let $\Phi$ be the root system of type $A_{n-1}$ as in Example \ref{An}. Moreover, let $(\tilde W, \tilde S)$ be an affine Coxeter system of type $\tilde A_{n-1}$. Furthermore, let 
\begin{itemize}
\item $\mathbb{F}$ be a field, 
\item $G = \mathrm{GL}_n(\mathbb{F}[t,t^{-1}])$, considered as the group of invertible $(n \times n)$-matrices over the ring of Laurent polynomials $\mathbb{F}[t,t^{-1}]$,
\item $T := \{ \mathrm{diag}(\lambda_1,...,\lambda_n) \mid \lambda_i \in \mathbb{F}^* \}$ be the group of invertible diagonal $(n \times n)$-matrices over $\mathbb{F}$,
\item $U_{\alpha_{i,j},d} := \{ 1_{n \times n} + \lambda t^d e_{ij} \mid \lambda \in \mathbb{F} \}$, where $d \in \mathbb{Z}$, and $1_{n \times n}$ denotes the identity $(n \times n)$-matrix and $e_{ij}$ the matrix whose $(i,j)$-entry is $1$ and all other entries are $0$.
\end{itemize}
Then the triple $(G,\{ U_{\alpha_{i,j},d} \}_{\alpha_{i,j} \in \Phi, d \in \mathbb{Z}}, T)$ is an $\mathbb{F}$-locally split RGD system of type $(\tilde W, \tilde S)$.
Again, replacing $G$ by $G' = \mathrm{SL}_n(\mathbb{F}[t,t^{-1}])$ and $T$ by $T' := \{ \mathrm{diag}(\lambda_1,...,\lambda_{n-1},\lambda_1^{-1}\cdots\lambda_{n-1}^{-1}) \mid \lambda_i \in \mathbb{F}^* \}$ also yields an $\mathbb{F}$-locally split RGD system $(G',\{ U_{\alpha_{i,j},d} \}_{\alpha_{i,j} \in \Phi, d \in \mathbb{Z}}, T')$ of type $(\tilde W, \tilde S)$; cf.\ \cite[Section~2.7]{Caprace/Remy:2008-LN}, \cite[page~17, Example~2]{Abramenko:1996}, \cite[Section~8.3.2]{Abramenko/Brown:2008}. 
\end{example}

\subsubsection*{(Twin) $BN$-pairs}

Following \cite[Section 6.2]{Abramenko/Brown:2008} (also \cite[Chapter IV]{Bourbaki:1968}, \cite{Bourbaki:2002}, \cite{Tits:1974}) a pair of subgroups $B$ and $N$ of a group $G$ is called a \Defn{$BN$-pair}, if $B$ and $N$ generate $G$, the intersection $T:=B\cap N$ is normal in $N$, and the quotient group $W:=N/T$ admits a set $S$ of generators such that the following conditions hold:
\begin{axioms}{BN}
\item $wBs \subseteq BwsB\cup BwB$ for all $w \in W$, $s \in S$, and
\item $sBs^{-1} \not\subset B$ for all $s \in S$.
\end{axioms}
The group $W$ is called the \Defn{Weyl group} associated to the $BN$-pair.
The quadruple $(G,B,N,S)$ is also called a \Defn{Tits system}.

\medskip
A group $G$ with a $BN$-pair admits a Bruhat decomposition $G = \bigsqcup_{w \in W} BwB$; cf.\ \cite[Theorems 6.17 and 6.56]{Abramenko/Brown:2008}. For each $J \subset S$ the set $P_J := \bigsqcup_{w \in W_J} BwB$ is a subgroup of $G$, the \Defn{standard parabolic subgroup} of type $J$. The pair $(W,S)$ is a Coxeter system.

\medskip
Let $(G,B_+,N,S)$ and $(G,B_-,N,S)$ be two Tits systems such that $B_+ \cap N=B_- \cap N$. Following \cite[Section 3.2]{Tits:1992}, \cite[Section 6.3.3]{Abramenko/Brown:2008} the triple $(B_+,B_-,N)$ is called a \Defn{twin $BN$-pair}, if the following conditions are satisfied:
\begin{axioms}{TBN}
\item $B_\eps w B_{-\eps} s B_{-\eps} = B_\eps w s B_{-\eps}$ for $\eps \in \{+,-\}$ and all $w \in W$, $s \in S$ such that $l(ws) < l(w)$, and
\item $B_+s \cap B_- = \emptyset$ for all $s \in S$.
\end{axioms}
A twin $BN$-pair is called \Defn{saturated}, if $B_+ \cap B_- = T$.
A group $G$ with a twin $BN$-pair admits the {Birkhoff decomposition} 
$G = \bigsqcup_{w \in W} B_{\eps} wB_{-\eps}$, where $\eps\in \{+,-\}$; cf.\ \cite[Proposition 6.81]{Abramenko/Brown:2008}.
The tuple $(G,B_+,B_-,N,S)$ is called a \Defn{twin Tits system}.

\begin{remark}
By \cite[Lemma~6.85]{Abramenko/Brown:2008}, if $(B_+,B_-,N)$ is a saturated twin $BN$-pair in a group $G$ and if $N'$ is an arbitrary subgroup of $G$, then $(B_+,B_-,N')$ is a twin $BN$-pair if and only if $N'(B_+ \cap B_-)=N$. Moreover, by the same lemma, if a pair $B_+$, $B_-$ of subgroups of $G$ is part of a $BN$-pair, then there is a unique subgroup $N \leq G$ such that $(B_+,B_-,N)$ is a saturated twin $BN$-pair. It is therefore only a very mild condition for a twin $BN$-pair to be saturated. 
\end{remark}

By \cite[Proposition 4]{Tits:1992}, \cite[Theorem 8.80]{Abramenko/Brown:2008}, if $(G,\Uals,T)$ is an RGD system of type $(W,S)$, then for
\begin{eqnarray*}
N & := & T.\gen{\mu(u) \mid u \in U_\alpha \backslash \{ 1 \}, \alpha \in \Pi }, \\
B_+ & := & T.U_+, \\
B_- & := & T.U_-,
\end{eqnarray*}
the tuple $(G,B_+,B_-,N,S)$ is a saturated twin $BN$-pair of $G$ with Weyl group $N/T\isomorph W$. It is called the twin $BN$-pair \Defn{associated} to the root group datum.

\begin{example} \label{An2}
Again, let $n \geq 2$, let $(W,S)$ be the Coxeter system of type $A_{n-1}$, and let $\Phi$ be the root system $\Phi$ of type $(W,S)$. As discussed in Example \ref{An}, $\Phi$  has one root $\alpha_{i,j}$ for each ordered pair $(i,j)$ with $1 \leq i, j \leq n$ and $i \neq j$. Define $\Phi_+ := \{ \alpha_{i,j} \in \Phi \mid i < j \}$ and $\Phi_- := \{ \alpha_{i,j} \in \Phi \mid i > j \}$.
Moreover, also as in Example \ref{An}, let
\begin{itemize}
\item $\mathbb{F}$ be a field, 
\item $G = \mathrm{GL}_n(\mathbb{F})$, considered as the group of invertible $(n \times n)$-matrices over $\mathbb{F}$,
\item $T := \{ \mathrm{diag}(\lambda_1,...,\lambda_n) \mid \lambda_i \in \mathbb{F}^* \}$,
\item $U_{\alpha_{i,j}} := \{ 1_{n \times n} + \lambda e_{ij} \mid \lambda \in \mathbb{F} \}$.
\end{itemize}
Then for
\begin{itemize}
\item $U_+ = \langle U_{\alpha_{i,j}} \mid \alpha_{i,j} \in \Phi_+ \rangle = \langle U_{\alpha_{i,j}} \mid \alpha_{i,j} \in \Phi, i < j \rangle$ is the group of strictly upper triangular $(n \times n)$-matrices over $\mathbb{F}$,
\item $B_+ = T.U_+$ is the group of upper triangular $(n \times n)$-matrices over $\mathbb{F}$,
\item $U_- = \langle U_{\alpha_{i,j}} \mid \alpha_{i,j} \in \Phi_- \rangle = \langle U_{\alpha_{i,j}} \mid \alpha_{i,j} \in \Phi, i > j \rangle$ is the group of strictly lower triangular $(n \times n)$-matrices over $\mathbb{F}$,
\item $B_- = T.U_-$ is the group of lower triangular $(n \times n)$-matrices over $\mathbb{F}$,
\item $N$ is the group of invertible monomial $(n \times n)$-matrices over $\mathbb{F}$
\end{itemize}
the tuple $(G,B_+,B_-,N,S)$ is a twin Tits system, and the tuple $(B_+,B_-,N)$ is a saturated twin $BN$-pair.
\end{example}

\begin{example} \label{Antilde2}
Once more, let $n \geq 2$, let $(W,S)$ be the Coxeter system of type $A_{n-1}$, and let $\Phi$ be the root system of type $A_{n-1}$. As in Example \ref{Antilde} let
\begin{itemize}
\item $\mathbb{F}$ be a field, 
\item $G = \mathrm{GL}_n(\mathbb{F}[t,t^{-1}])$, considered as the group of invertible $(n \times n)$-matrices over the ring of Laurent polynomials $\mathbb{F}[t,t^{-1}]$,
\item $T := \{ \mathrm{diag}(\lambda_1,...,\lambda_n) \mid \lambda_i \in \mathbb{F}^* \}$,
\item $U_{\alpha_{i,j},d} := \{ 1_{n \times n} + \lambda t^d e_{ij} \mid \lambda \in \mathbb{F} \}$, where $d \in \mathbb{Z}$.
\end{itemize}
Then for
\begin{itemize}
\item $B_+ = T.U_+ = T \langle U_{\alpha_{i,j},d} \mid \alpha_{i,j} \in \Phi, d \in \mathbb{Z}, (i < j \Rightarrow d \geq 0), (i > j \Rightarrow d > 0) \rangle$ is the group of invertible $(n \times n)$-matrices over the ring of polynomials $\mathbb{F}[t]$ which modulo $t$ are upper triangular,
\item $B_- = T.U_- = T \langle U_{\alpha_{i,j},d} \mid \alpha_{i,j} \in \Phi, d \in \mathbb{Z}, (i < j \Rightarrow d < 0), (i > j \Rightarrow d \leq 0) \rangle$ is the group of invertible $(n \times n)$-matrices over the ring of polynomials $\mathbb{F}[t^{-1}]$ which modulo $t^{-1}$ are lower triangular,
\item $N$ is the group of invertible monomial $(n \times n)$-matrices over the ring of Laurent polynomials $\mathbb{F}[t,t^{-1}]$,
\end{itemize}
the tuple $(G,B_+,B_-,N,S)$ is a twin Tits system.
\end{example}

\subsubsection*{Chamber systems}

Let $I$ be a finite set. A \Defn{chamber system} over $I$ is a pair
\mbox{$\mcC= (C, (\sim_i)_{i \in I})$}, where $C$ is a non-empty set
whose elements are called \Defn{chambers} and where each $\sim_i$, $i \in I$, is an equivalence relation on the set of chambers; see \cite{Tits:1981}, \cite[Section 5.2]{Abramenko/Brown:2008}, \cite[Chapter 3]{Buekenhout/Cohen}

The \Defn{rank} of a chamber system of type $I$ is the cardinality of $I$.
Given $i \in I$ and $c,d \in C$,
then $c$ is called \Defn{$i$-adjacent} to $d$, if $c \sim_i d$.
The chambers $c$, $d$ are called \Defn{adjacent}, if
they are $i$-adjacent for some $i \in I$.

A \Defn{gallery} in $\mcC$ is a finite sequence $(c_0,c_1,\ldots,c_k)$  such that $c_{\mu} \in C$  for all $0 \leq \mu \leq k$ and such that $c_{\mu-1}$ is adjacent to $c_{\mu}$ for all $1 \leq \mu \leq k$. The number $k$ is called the \Defn{length} of the gallery. 

For a subset $J \subseteq I$ a \Defn{$J$-gallery} is a gallery 
$G =  (c_0,c_1,\ldots,c_k)$ such that for each $1 \leq \mu \leq k$
there exists an index $j \in J$ with
$c_{\mu - 1} \sim_j c_{\mu}$.
Given two chambers $c$, $d$, we say that $c$ is \Defn{$J$-equivalent} to $d$, if there exists a $J$-gallery joining $c$ and $d$; we write $c \sim_J d$ in this case.
Note that since $\sim_i$ is an equivalence relation, $c$ and $d$ are $i$-adjacent if and only if they are $\{ i \}$-equivalent.

Given a chamber $c$ and a subset $J$ of $I$, the set
$R_J(c) := \{ d \in C \mid c \sim_J d \}$ is called the
\Defn{$J$-residue} of $c$.
If $J = \{ i \}$, then $R_J(c)$ is called the \Defn{$i$-panel} of $c$ (or the $i$-panel containing $c$); a \Defn{panel} is an $i$-panel for some $i \in I$.
Note that $(R_J(c), (\sim_j)_{j \in J})$ is a connected chamber system over $J$.

Let $\mcC$ be a chamber system over $I$ and let $K \subseteq I$. Then the \Defn{$K$-residue chamber system} $\mcC_K$ over $I':=I\setminus{K}$ is defined as follows: Its chambers are the $K$-residues $R_K(c)$ in $\mcC$ and, for $i\in I'$, two chambers $R_K(c)$, $R_K(d)$ of $\mcC_K$ are $i$-adjacent if and only if both are contained in the same $(K\cup\{i\})$-residue. In other words, for $J \subseteq I'$, the $J$-residues of $\mcC_K$ are precisely the $(J \cup K)$-residues of $\mcC$.

A chamber system $\mcC$ over $I$ is called \Defn{residually connected} (see \cite[Chapter 3]{Buekenhout/Cohen}), if for each $J \subseteq I$ and for every family of residues $(R_{I\setminus\{j\}})_{j\in J}$ with pairwise non-trivial intersection, the intersection $\bigcap_{j\in J} R_{I\setminus\{j\}}$ is an $(I\setminus J)$-residue, i.e., $\bigcap_{j\in J} R_{I\setminus\{j\}}$ is a non-empty connected induced chamber system over $I\setminus J$. 

Let $\mcC=(C, (\sim_i)_{i \in I})$ be a chamber system and let $C' \subseteq C$. The chamber system $\mcC' = (C',({\sim_i}_{|C' \times C'})_{i \in I})$ is called the sub-chamber system of $\mcC$ \Defn{induced} on $C'$.
If $\mcC=(C, (\sim_i)_{i \in I})$ is a chamber system with an induced sub-chamber system $\mcC'= (C',({\sim_i}_{|C' \times C'})_{i \in I})$ and a distinguished subset $X \subseteq C'$, the chamber system $\mcC'$ \Defn{inherits connectedness from $\mcC$ within $X$}, if two chambers $c,d \in X$ are connected by a $J$-gallery in $\mcC'$ if and only if they are connected by a $J$-gallery in $\mcC$. 
If $\mcC'$ inherits connectedness from $\mcC$ within $\mcC'$, then one also says that $\mcC'$ \Defn{inherits connectedness from $\mcC$}.

\subsubsection*{(Twin) buildings}

Let $(W,S)$ be a Coxeter system.
A \Defn{building} of type $(W,S)$ is a pair $(\mcC,\delta)$ where $\mcC$ is a non-empty set and
$\delta : \mcC \times \mcC \to W$ is a \Defn{distance
function}
satisfying the following axioms, where $x,y \in  \mcC$ and
$w = \delta(x,y)$:

\begin{axioms}{Bu}
\item $w = 1$ if and only if $x = y$,
\item if $z \in  \mcC$ is such that $\delta(y,z) = s \in S$,
then $\delta(x,z) \in \{ w, ws \}$, and if furthermore $l(ws) = l(w) + 1$,
then $\delta(x,z) = ws$, and
\item if $s \in S$, there exists $z \in  \mcC$ such that
$\delta(y,z) = s$ and $\delta(x,z) = ws$.
\end{axioms}
For detailed information on buildings refer to \cite{Tits:1974}, \cite{Ronan:1989}, \cite{Abramenko/Brown:2008}.

\begin{remark} \label{chambersystembuilding}
For a building $(\mcC,\delta)$ of type $(W,S)$ and $s\in S$, we define a relation $\sim_s$, where $c,d\in\mcC$ are $s$-equivalent, i.e., $c\sim_s d$, if and only if $\delta(c,d)\in\{1_W,s\}$. From the axioms above it follows that this is in fact an equivalence relation, and $( \mcC,(\sim_s)_{s \in S})$ is a chamber system (see \cite[Section 5.1.1]{Abramenko/Brown:2008}).
By \cite[Theorem 2]{Tits:1981} it is possible to reconstruct the building and its distance function from this chamber system and, hence, one does not need to distinguish between the building and its chamber system. In particular, one may speak of galleries, residues and panels of a building.
\end{remark}

\begin{example} \label{An3}
Let $n \geq 2$, let $(W,S)$ be a Coxeter system of type $A_{n-1}$, i.e., as in Example \ref{An} let $W \cong S_n$ be the symmetric group on $n$ letters and let $S = \{ s_1, ..., s_{n-1} \}$ be its standard generating set consisting of adjacent transpositions, let $\mathbb{F}$ be a field, and let $V$ be an $n$-dimensional vector space over $\mathbb{F}$. Define $\mathcal{C} := \{ (V_1, ..., V_{n-1}) \subset V^{n-1} \mid \{ 0 \} \neq V_1 \lneq V_2 \lneq \cdots \lneq V_{n-1} \lneq V \}$ to be the set of maximal flags of non-trivial proper $\mathbb{F}$-subspaces of $V$. For $1 \leq i \leq n-1$, define two maximal flags $(V_1, ..., V_{n-1})$ and $(V_1', ..., V_{n-1}')$ to be $s_i$-adjacent if and only if for all $1 \leq j \leq n-1$ with $j \neq i$ one has $V_j = V_j'$. Then $(\mathcal{C},(\sim_{s_i})_{1 \leq i \leq n-1})$, where $\sim_{s_i}$ denotes the equivalence relation of $s_i$-adjacency, is a chamber system, which gives rise to a building of type $(W,S)$ via Remark \ref{chambersystembuilding}; cf.\ \cite[Definition~4.25]{Abramenko/Brown:2008}.
\end{example}

\begin{remark}\label{Titssystem}
A Tits system $(G, B, N, S)$ leads to a building whose set of chambers equals $G/B$ and whose distance function $\delta : G/B \times G/B \to W$ is given by $\delta(gB,hB) = w$ if and only if $Bg^{-1}hB = BwB$; see \cite[p.\ 320, item (4)]{Abramenko/Brown:2008}. 
\end{remark}

\begin{example} \label{An35}
Let $G = \mathrm{GL}_n(\mathbb{F})$ with its Tits system $(G,B_+,N,S)$ from Example \ref{An2} and let $(\mathcal{C},(\sim_{s_i})_{1 \leq i \leq n-1})$ be the chamber system of the building of type $(W,S)$ from Example \ref{An3} arising from the $\mathbb{F}$-vector space $V \cong \mathbb{F}^n$. 
After choosing a basis $\{ e_1, \ldots, e_n \}$ of $V$, the matrix group $G$ acts on $V$. This induces a transitive action of $G$ on the set of chambers $\mathcal{C}$ which preserves each of the equivalence relations $\sim_{s_i}$. Therefore, as the stabilizer in $G$ of the maximal flag $c := (\langle e_1 \rangle, \langle e_1, e_2 \rangle \ldots, \langle e_1, e_2, \ldots, e_{n-1} \rangle)$ equals the group $B_+$ of upper triangular matrices, by the fundamental theorem of permutation representations one indeed has $G/B_+ = G/\mathrm{Stab}_G(c) \cong \mathcal{C}$.
For $1 \leq i \leq n-1$, the $s_i$-panel of $c$ is 
\begin{eqnarray*}
& & R_{\{s_i\}}(c) \\ & = & \{ (\langle e_1 \rangle, \ldots, \langle e_1, ..., e_{i-1} \rangle, V_i, \langle e_1, ..., e_{i+1} \rangle, \ldots, \langle e_1, \ldots, e_{n-1} \rangle) \mid \langle e_1, ..., e_{i-1} \rangle \lneq V_i \lneq \langle e_1, ..., e_{i+1} \rangle \}.
\end{eqnarray*}
The stabilizer in $G$ of this $s_i$-panel equals $\bigsqcup_{w \in \langle s_i \rangle \subseteq W} B_+ w B_+$. Therefore, $gB_+ \in G/B_+ \cong \mathcal{C}$ is $s_i$-adjacent to $B_+$ (which corresponds to the chamber $c \in \mathcal{C}$) if and only if $g \in \bigsqcup_{w \in \langle s_i \rangle \subseteq W} B_+ w B_+$ which is the case if and only if $B_+gB_+ \subseteq \bigsqcup_{w \in \langle s_i \rangle \subseteq W} B_+ w B_+$. More precisely, one has $B_+gB_+ = B_+$ if and only if $g \in B_+$ and $B_+gB_+ = B_+s_iB_+$ if and only if $g \in \bigsqcup_{w \in \langle s_i \rangle \subseteq W} B_+ w B_+ \backslash B_+$.
\end{example}

\begin{remark} \label{simplicial}
Buildings can also be considered as simplicial complexes, see \cite[Chapter~4]{Abramenko/Brown:2008}. In this article we will sometimes implicitly switch to that point of view.
\end{remark}

The \Defn{rank} of a building of type  $(W,S)$ is $\abs{S}$.
A building is \Defn{thick} (resp.\ \Defn{thin}), if for any $s \in S$ and any chamber $c \in C$ there are at least three (resp.\ exactly two) chambers $s$-adjacent to $c$.
The numerical distance $l(x,y)$ of two chambers $x$ and $y$ is defined as $l(\delta(x,y))$. A building is called \Defn{spherical}, if its Coxeter system $(W,S)$ is spherical, i.e., if $W$ is finite. In a spherical building, two chambers $c,d$ are called \Defn{opposite}, if $\delta(c,d)=w_S$, the longest element of $(W,S)$.

\begin{example} \label{exmpl:coxeter-building}
Let $(W,S)$ be a Coxeter system. Define $\delta_S:W\times W\to W: (x,y)\mapsto x^{-1}y$. Then $\delta_S$ is a distance function and $(W,\delta_S)$ is a thin building of type $(W,S)$. Any thin building of type $(W,S)$ is isometric to the one given here, cf.\ \cite[Exercise 4.12]{Abramenko/Brown:2008}.
\end{example}

Suppose $(\mcC,\delta)$ is a building of type $(W,S)$. Then an \Defn{apartment} of $\mcC$ is a subset $\Sigma$ of $\mcC$, such that $(\Sigma,\delta|_\Sigma)$  is isometric to $(W,\delta_S)$ (cf.\ Example \ref{exmpl:coxeter-building}).
A \Defn{root} of an apartment $\Sigma$ (considered as a simplicial complex; cf.\ Remark \ref{simplicial}) is a subcomplex $\alpha \subseteq \Sigma$ that is the image of a reversible folding in the sense of \cite[Definition~3.48]{Abramenko/Brown:2008}. The subcomplex $-\alpha$ of $\Sigma$ generated by the chambers not in $\alpha$ is again a root, the root \Defn{opposite} to $\alpha$. A subcomplex $\alpha$ of a building (considered as a simplicial complex; cf.\ Remark \ref{simplicial}) is called a \Defn{root}, if there is an apartment $\Sigma$ such that $\alpha \subseteq \Sigma$ and $\alpha$ is a root in $\Sigma$. According to \cite[Exercise~4.61]{Abramenko/Brown:2008}, if $\alpha$ is a root in a building, then it is a root in every apartment containing it.

\begin{example} \label{An4}
Let $n \geq 2$, let $(W,S)$ be a Coxeter system of type $A_{n-1}$, let $\mathbb{F}$ be a field, let $V$ be an $n$-dimensional vector space over $\mathbb{F}$, and let $(\mathcal{C},(\sim_{s_i})_i)$ be the chamber system of the building of type $A_{n-1}$ described in Example \ref{An3}.
Moreover, let $\{ e_1, ..., e_n \}$ be a basis of $V$ and for $m \in \mathbb{N}$ define $\mathbf{m} := \{ 1, ..., m \}$. Then the set $$\{ (V_1, ..., V_{n-1}) \in \mathcal{C} \mid \mbox{ $\forall i \in \mathbf{n-1} \exists J_i \subseteq \mathbf{n}$ with $|J_i| = i$ such that $V_i = \langle e_j \mid j \in J_i \rangle$ } \}$$
defines an apartment of that building. In fact, to each apartment $\Sigma$ of the building $(\mathcal{C},(\sim_{s_i})_i)$ there exists a basis of $V$ such that $\Sigma$ can be described as above.
\end{example}

By \cite[Proposition 5.34]{Abramenko/Brown:2008} for each residue $R$ and chamber $d$ of a building $\mcC$, there exists a unique $c \in R$ such that $l(c,d) = \mathrm{min} \{l(x,d) \mid x \in R\}$, called the \Defn{projection of $d$ onto $R$} and denoted by {$\proj_R d$}. If $S$ is another residue of $\mcC$, define $\proj_R S := \{ \proj_R d\mid d\in S\}$, the \Defn{projection of $S$ onto $R$}. Note that the subset $\proj_R S$ of $R$ by \cite[Lemma 5.36]{Abramenko/Brown:2008} is a residue of $\mcC$.

\medskip
A \Defn{twin building} of type $(W,S)$ is a triple
$\mcC = ((\mcC_+,\delta_+),(\mcC_-,\delta_-),\delta^*)$ consisting of two buildings $(\mcC_+,\delta_+)$ and $(\mcC_-,\delta_-)$ of type $(W,S)$ together with a \Defn{codistance} function
$\delta^*: ( \mcC_+ \times  \mcC_-) \cup (\mcC_- \times \mcC_+)
\rightarrow W$
satisfying the following axioms, where $\eps \in \{ +,- \}$,
$x \in \mcC_{\eps}$, $y \in  \mcC_{-\eps}$, and
$w = \delta^*(x,y)$:
\begin{axioms}{Tw}
\item \label{axioms:Tw1} $\delta^*(y,x) = w^{-1}$,
\item \label{axioms:Tw2} if $z \in  \mcC_{-\eps}$ is such that
$\delta_{-\eps}(y,z)=s \in S$ and $l(ws) = l(w) - 1$,
then $\delta^*(x,z) = ws$, and
\item \label{axioms:Tw3} if $s \in S$, there exists $z \in  \mcC_{-\eps}$
such that $\delta_{-\eps}(y,z) = s$ and $\delta^*(x,z) = ws$.
\qedhere
\end{axioms}

Two chambers $c\in \mcC_\eps$ and $d\in \mcC_{-\eps}$ ($\eps \in \{ +,- \}$) are called \Defn{opposite}, in symbols $c \op d$, if $\delta^*(c,d)=1_W$.
Two residues are \Defn{opposite}, if they have the same type and contain opposite chambers.
A \Defn{twin apartment} of a twin building $\mcC$ is a pair $\Sigma=(\Sigma_+,\Sigma_-)$ such that $\Sigma_+$ is an apartment of $\mcC_+$, $\Sigma_-$ is an apartment of $\mcC_-$, and every chamber in $\Sigma_+\cup \Sigma_-$ is opposite to precisely one other chamber in $\Sigma_+\cup \Sigma_-$; this induces the \Defn{opposition involution} $\mathrm{op}_\Sigma : \Sigma \to \Sigma$ which maps a chamber of $\Sigma$ to its opposite. 
A \Defn{twin root} of a twin apartment $\Sigma$ is a pair $\alpha = (\alpha_+,\alpha_-)$ of roots in $\Sigma$ such that $\mathrm{op}_\Sigma(\alpha_+) = -\alpha_-$ and $\mathrm{op}_\Sigma(\alpha_-) = -\alpha_+$; cf.\ \cite[Definition~5.190]{Abramenko/Brown:2008}. 
Given a twin root $\alpha$ and a panel $P$ of $\mcC_+$ or $\mcC_-$ that meets $\alpha$, the panel $P$ is called a \Defn{boundary panel} of $\alpha$, if $P \cap \alpha$ consists of exactly one chamber; otherwise, $P$ is called an \Defn{interior panel} of $\alpha$. If $P$ is a boundary panel of $\alpha$, define $\mcC(P,\alpha) := P \backslash \{ c \}$, where $c$ is the chamber in $P \cap \alpha$.

\begin{remark} \label{rmk:spherical-flips-1}
By \cite[Proposition 1]{Tits:1992} (also \cite[Example 5.136]{Abramenko/Brown:2008}) there is a canonical one-to-one correspondence between spherical buildings of type $(W,S)$ and spherical twin buildings of type $(W,S)$: For a spherical building $(\mcC,\delta)$, the tuple $\mcC_{\mathrm{twin}} := ((\mcC,\delta),(\mcC_-,\delta_-),\delta^*)$ with $\mcC_- := \{ c_- \mid c \in \mcC \}$ and functions $\delta_- : \mathcal{C}_- \times \mcC_- \to W : (c_-,d_-) \mapsto w_S\delta(c,d)w_S$ and $$\delta^* : (\mcC \times \mcC_-) \cup (\mcC_- \times \mcC) \to W : \left\{ \begin{matrix} (c,d_-) \mapsto \delta(c,d)w_S \\ (c_-,d) \mapsto w_S\delta(c,d) \end{matrix} \right.$$
is a twin building, and any twin building with $(\mcC,\delta)$ as a positive half is isomorphic to this twin building.
Two chambers $c,d \in \mcC$ are opposite (i.e., $\delta(c,d)=w_S$) if and only if $c$ and $d_-$ are opposite (i.e., $\delta^*(c,d_-)=\delta(c,d)w_S = 1_W$) if and only if $c_-$ and $d$ are opposite (i.e., $\delta^*(c_-,d)=w_S\delta(c,d) = 1_W$).
\end{remark}

By \cite[Lemma 5.149]{Abramenko/Brown:2008},
if $R$ is a residue in $\mcC_\eps$ of spherical type, and $d$ is a chamber in $\mcC_{-\eps}$, then there is a unique chamber $c'\in R$ such that $\delta^*(c',d)$ is of maximal length in $\delta^*(R,d)$. This chamber satisfies $\delta^*(c,d) = \delta_\eps(c,c')\delta^*(c',d)$
for all $c\in R$. This chamber $c'$ is called the \Defn{projection} of $d$ onto $R$, in symbols $\proj_R(d)$.

A pair $(M_+,M_-)$ of non-empty subsets $M_+\subseteq \mcC_+$ and $M_-\subseteq \mcC_-$ is called \Defn{convex}, if $\proj_P c\in M_+\cup M_-$ for any $c\in M_+\cup M_-$ and any panel $P\subseteq \mcC_+\cup\mcC_-$ that meets $M_+\cup M_-$.
Two spherical residues $R$ and $Q$ are called \Defn{parallel} if $\proj_R(Q)=R$ and $\proj_Q(R)=Q$.

\subsubsection*{Group actions on buildings}

Following \cite[Definition~6.67]{Abramenko/Brown:2008} a group $G$ \Defn{acts} on a twin building $\mcC= ((\mcC_+,\delta_+),(\mcC_-,\delta_-),\delta^*)$, if it acts simultaneously on the two sets $\mcC_+$ and $\mcC_-$ and preserves the distances $\delta_+$, $\delta_-$ and the codistance $\delta^*$. It acts \Defn{strongly transitively}, if it is transitive on the set $\{ (c_+,c_-) \in \mcC_+ \times \mcC_- \mid \delta^*(c_+,c_-) = 1_W \}$ of pairs of opposite chambers of $\mcC$; see \cite[Lemma~6.70]{Abramenko/Brown:2008} for a collection of characterizations of strong transitivity. By \cite[Corollary~6.79]{Abramenko/Brown:2008}, a group acting strongly transitively on a thick twin building admits a twin $BN$-pair.

For any twin root $\alpha$ of a twin building $\mcC$, the \Defn{root group} $U_\alpha$ is defined to be the set of type-preserving automorphisms $g$ of $\mcC$ such that $g$ fixes $\alpha$ and every interior panel of $\alpha$ pointwise. The twin building $\mcC$ is called \Defn{Moufang}, if for each twin root $\alpha$ and each boundary panel $P$, the action of the root group $U_\alpha$ on the set $\mcC(P,\alpha) = P \backslash \{ c \}$ is transitive, where $c$ is the chamber in $P \cap \alpha$; it is called \Defn{strictly Moufang}, if these actions are simply transitive, see \cite[Section~8.3]{Abramenko/Brown:2008}. Any Moufang twin building that does not admit a root group $U_\alpha$, $\alpha \in \Phi$, which commutes with each root group $U_\beta$, $\beta \in \Phi \backslash \{ \alpha, -\alpha \}$, is strictly Moufang by \cite[Proposition 7.79]{Abramenko/Brown:2008}.

\medskip
Conversely, a group $G$ with a twin $BN$-pair yields two buildings $(G/B_+,\delta_+)$ and $(G/B_-,\delta_-)$ with distance functions $\delta_\eps$, $\eps \in \{ +,- \}$, defined as $\delta_\eps : G/B_\eps \times G/B_\eps \to W$ via $\delta_\eps(gB_\eps,hB_\eps) = w$ if and only if $B_\eps h^{-1}gB_\eps = B_\eps wB_\eps$ using the Bruhat decomposition; cf.\ Remark \ref{Titssystem}. Furthermore, using the Birkhoff decomposition one can define the codistance function $\delta^* : (G/B_- \times G/B_+) \cup (G/B_+ \times G/B_-) \to W$ via $\delta^*(gB_-,hB_+) = w$ if and only if $B_+h^{-1}gB_- = B_+wB_-$ and $\delta^*(hB_+,gB_-) := (\delta^*(gB_-,hB_+))^{-1}$. The tuple $((G/B_+,\delta_+),(G/B_-,\delta_-),\delta^*)$ then is a twin building, the \Defn{twin building associated to $G$}\index{associated twin building}; see \cite[Theorem~6.87]{Abramenko/Brown:2008}.

\medskip
Hence, to every RGD system, a twin building is associated in a natural way (recall \cite[Proposition 4]{Tits:1992}, \cite[Theorem 8.80]{Abramenko/Brown:2008}).
An RGD system $(G,\Uals,T)$ is called \Defn{faithful}, if $G$ operates faithfully on the associated building. It is called \Defn{centered}, if $G$ is generated by its root groups, and \Defn{reduced}, if it is both centered and faithful.
For any RGD system $(G,\Uals,T)$, denote by $G^\circ$ the quotient of the subgroup $\gen{U_\alpha\mid \alpha\in\Phi}$ by its center, and by $U^\circ_\alpha$ the canonical image of $U_\alpha$ in $G^\circ$. Unless there exists $\alpha \in \Phi$ orthogonal to all other roots, the canonical homomorphisms $U_\alpha\to U^\circ_\alpha$ are isomorphisms. Then $(G^\circ,\{U^\circ_\alpha\}_{\alpha\in\Phi})$ is a reduced RGD system with the same associated twin building as $(G,\Uals,T)$, which is called its \Defn{reduction}.

\medskip
 A \Defn{Moufang set} is a set $X$ containing at least two elements together with a collection of \Defn{root groups} $(U_x)_{x \in X}$ such that each $U_x$ is a subgroup of $\Sym(X)$ fixing $x$ and acting sharply transitively on $X \setminus \{ x \}$ and such that each $U_x$
        permutes the set $\{ U_y \mid y \in X \backslash \{ x \}\}$ by conjugation.
        The group $\langle U_x \mid x \in X \rangle$ generated by the root groups $U_x$ is called the
        \Defn{little projective group} of the Moufang set.
An \Defn{automorphism} of a Moufang set $\mouf = (X,(U_x)_{x \in X})$ is a permutation of the set $X$ such that the induced inner automorphism of the group $\mathrm{Sym}(X)$ permutes the set $\{ U_x \mid x \in X \}$. 

\begin{remark}
Moufang sets are closely related to rank one groups; see \cite{Medts/Segev:2009}, \cite{Timmesfeld:2001}, \cite[Section~2.5]{Caprace/Remy:2008-LN}.
\end{remark}

\begin{example}
Consider the $\mathbb{F}$-locally split RGD system $$\left(G,\{ U_\alpha, U_{-\alpha} \}, \left\{ \begin{pmatrix} \lambda_1 & 0 \\ 0 & \lambda_2 \end{pmatrix} \mid \lambda_1, \lambda_2 \in \mathbb{F}^* \right\} \right)$$ of the matrix group $G = \mathrm{GL}_2(\mathbb{F})$ described in Example \ref{An}. The building of $G$ arising from this RGD system is isomorphic to the projective line $X := G/B_+$; see Examples \ref{An2}, \ref{An3}, \ref{An35}. For each $x = gB_+ \in G/B_+$ define $U_x := R_u(gB_+g^{-1}) = gU_+g^{-1}$. Then $\mouf = (X,(U_x)_{x \in X})$ is a Moufang set. Its little projective group is $\mathrm{SL}_2(\mathbb{F})$.
\end{example}

\section{Flips and quasi-flips} \label{sec:flip-intro} \label{sec:flip-correspondence}
%  ___         _   _          
% / __| ___ __| |_(_)___ _ _  
% \__ \/ -_) _|  _| / _ \ ' \ 
% |___/\___\__|\__|_\___/_||_|

Let $G$ be a group that acts strongly transitively on a thick twin building and let $(B_+,B_-,N)$ be the twin $BN$-pair of $G$ resulting from that action; cf.\ \cite[Corollary~6.79]{Abramenko/Brown:2008}. Inspired by \cite{Bennett/Gramlich/Hoffman/Shpectorov:2003}, in this article article we study involutive automorphisms of $G$ that map $B_+$ onto $G$-conjugates of $B_-$.

\begin{prop} \label{prop:BN-quasi-flip-Weyl-auto} \label{prop:building-is-bn-flip}
  Let $G$ be a group with a saturated twin $BN$-pair $(B_+,B_-,N,S)$ of type $(W,S)$, let $T:=B_+\cap B_-$, and 
  let $\theta$ be an automorphism of $G$ satisfying $\theta^2 = \id$ such that there exists $g\in G$ with $\theta(B_+) = gB_-g^{-1}$. 
If the set of chambers fixed by $T$ of the twin building associated to $G$ equals the twin apartment containing $B_+$ and $B_-$, then the following hold:
  \begin{enumerate}
    \item \label{prop:BN-quasi-flip-Weyl-auto1} There exists $x\in G$ such that $\theta(B_\eps) = xB_{-\eps} x^{-1}$ and $\theta(x)x\in T$, where $\eps \in \{+,-\}$.
    \item \label{prop:BN-quasi-flip-Weyl-auto2} $\theta$ induces the automorphism $\theta_x : nT\mapsto x^{-1}\theta(n)xT$ of the Coxeter system $(NT/T,S) \cong (W,S)$ of order at most $2$.
\end{enumerate}
Let $\mcC$ be the twin building $((G/B_+,\delta_+),(G/B_-,\delta_-),\delta^*)$ arising from the twin $BN$-pair of $G$.
\begin{enumerate}
\setcounter{enumi}{2}
\item \label{prop:BN-quasi-flip-Weyl-auto3} The map $\wtheta : \mcC \to \mcC : gB_{\eps} \mapsto \theta(g)xB_{-\eps}$ preserves adjacency and opposition. It furthermore satisfies
\begin{enumerate}
\item \label{prop:BN-quasi-flip-Weyl-auto3a} $\wtheta^2 = \mathrm{id}$ and
\item \label{prop:BN-quasi-flip-Weyl-auto3b} $\wtheta(gc)=\theta(g)\wtheta(c)$ for any $g\in G$, $c \in \mcC$.
\end{enumerate}
  \end{enumerate}
\end{prop}

\begin{remark}
We stress that the map $\tilde \theta$ in item (\ref{prop:BN-quasi-flip-Weyl-auto3}) of the preceding proposition not necessarily preserves the types of $\mcC$. In fact, it preserves types if and only if the automorphism $\theta_x$ of the group $NT/T$ in item (\ref{prop:BN-quasi-flip-Weyl-auto2}) is the identity automorphism.
\end{remark}

\begin{proof}
The normalizer $N_G(T)$ acts on the set of chambers fixed by $T$ of the twin building associated to $G$. Since by assumption this set equals the twin apartment containing $B_+$ and $B_-$ and since $N$ equals the full stabilizer of this twin apartment (as $(B_+,B_-,N,S)$ is saturated, cf.\ \cite[Definition~6.84]{Abramenko/Brown:2008}), the equality $N = N_G(T)$ holds.

(\ref{prop:BN-quasi-flip-Weyl-auto1})
By the Birkhoff decomposition, there exist $b_+\in B_+$, $b_-\in B_-$, $n\in N$ such that $\theta(g)g = b_+ n b_-$. Then 
\begin{eqnarray*}
\theta(gTg^{-1}) & = & \theta(g(B_+ \cap B_-)g^{-1}) \\
& = & \theta(g) \theta(B_+) \theta(g^{-1}) \cap \theta(gB_-g^{-1}) \\
& = & \theta(g)g B_- g^{-1}\theta(g)^{-1}  \cap B_+ \\
& = & (b_+ n b_-) B_- (b_+ n b_-)^{-1}  \cap B_+. 
\end{eqnarray*}
Hence, for $x:=b_+^{-1} \theta(g)$, we have $$x\theta(T)x^{-1} = b_+^{-1} \theta(gTg^{-1}) b_+ =  n B_- n^{-1}  \cap B_+ \geq T,$$ where the last containment holds because of $n \in N$. Therefore $$T \leq x\theta(T)x^{-1} \leq x\theta(x\theta(T)x^{-1})x^{-1} = x\theta(x)T(x\theta(x))^{-1}.$$ 
Accordingly, as $T$ fixes a unique twin apartment, $T=x\theta(x)T(x\theta(x))^{-1}$, i.e., $x\theta(x) \in N_G(T)$. We note $B_+ = x\theta(g)^{-1}b_+B_+b_+^{-1}\theta(g)x^{-1} = x\theta(B_-)x^{-1}$ and, thus, $$B_- \cap B_+ = T = x\theta(T)x^{-1} = x\theta(B_+)x^{-1} \cap x\theta(B_-)x^{-1} = x\theta(x) B_- (x\theta(x))^{-1} \cap  B_+.$$ 
Since $B_-$ is the unique chamber opposite $B_+$ in the twin apartment fixed by $T$, this means $x\theta(x)\in N_G(B_-)=B_-$ and, in particular, $\theta(x)B_- = x^{-1}B_-$. Therefore $x\theta(x) \in B_- \cap N_G(T) = T$ and $\theta(B_+)=\theta(x)B_-\theta(x)^{-1} = x^{-1}B_- x$.
 
(\ref{prop:BN-quasi-flip-Weyl-auto2})
 For each $x \in X:=\{x\in G \mid \theta(B_+) = x B_- x^{-1} \text{ and } \theta(B_-) = x B_+ x^{-1} \}$ the map $\theta_x:g\mapsto x^{-1} \theta(g) x$ preserves $T=B_+\cap B_-$. Therefore $\theta_x : W \to W : nT \mapsto \theta_x(nT)$ is an automorphism of $W = N/T$. It does not depend on the choice of $x$, because for $x'\in X$ we have $\theta_x(g)=x x'^{-1}\theta_{x'}(g) x' x^{-1}$ with $x x'^{-1} \in N_G(B_+) \cap N_G(B_-) = B_+ \cap B_- = T$.
 
The automorphism $\theta_x$ normalizes $S$, as for each $s = n_sT \in S \subseteq N/T$ the set $P_s := B_+ \cup B_+n_sTB_+$ is mapped by $\theta_x$ to $B_- \cup B_- \theta_x(n_s) B_-$. Therefore positive rank one parabolics are mapped to negative rank one parabolics, and $\theta_x$ acts on $S$. 

(\ref{prop:BN-quasi-flip-Weyl-auto3}) The map $\wtheta$ has order at most two, because
$\wtheta(\wtheta(gB_\eps))=\wtheta(\theta(g)xB_{-\eps})=\theta(\theta(g)x)xB_{\eps}=g\theta(x)xB_\eps=gB_\eps$, whence (\ref{prop:BN-quasi-flip-Weyl-auto3a}) holds. (\ref{prop:BN-quasi-flip-Weyl-auto3b}) is obvious.
By item (\ref{prop:BN-quasi-flip-Weyl-auto2}) the map $\theta_x : nT \mapsto \theta_x(nT)=\theta_x(n)T$ is an automorphism of the Weyl group $W=N/T$. We compute
 \begin{align*}
    \delta_{\eps} (g B_{\eps}, h B_{\eps}) = w
 \Longleftrightarrow {} &
	B_\eps g^{-1}hB_\eps = B_\eps w B_\eps \\
 \Longleftrightarrow {} &
      \theta(B_\eps) \theta(g^{-1}h) \theta(B_\eps)
    = \theta(B_\eps) \theta(w) \theta(B_\eps) \\
 \Longleftrightarrow {} &
      B_{-\eps} x^{-1} \theta(g^{-1})\theta(h) x B_{-\eps}
    = B_{-\eps} x^{-1} \theta(w) x B_{-\eps}
    = B_{-\eps} \theta_x(w) B_{-\eps} \\
 \Longleftrightarrow {} &
    \delta_{-\eps} (\theta(g)x B_{-\eps}, \theta(h)x B_{-\eps}) = \theta_x(w)
   .
 \end{align*}

Similarly we find that $\delta^* (g B_{\eps}, h B_{-\eps}) = w$ is equivalent to $\delta^* (\theta(g)x B_{-\eps}, \theta(h)x B_{\eps}) = \theta_x(w)$.
\end{proof}

\begin{remark}
If the group $G$ in Proposition \ref{prop:BN-quasi-flip-Weyl-auto} is endowed with a locally split RGD system over fields $(\mathbb{K}_\alpha)_{\alpha \in \Phi}$ satisfying $|\mathbb{K}_\alpha| \geq 4$ for each $\alpha \in \Phi$, then by \cite[Lemma~4.8]{Caprace:2009} the set of chambers fixed by the torus $T$ equals the twin apartment containing $B_+$ and $B_-$.
\end{remark}

\begin{remark}\label{thm:bn-is-building-flip}
Let $\mathcal{C} = (\mathcal{C}_+,\mathcal{C}_-,\delta^*)$ be a twin building with a strongly transitive group $G$ of automorphisms and let $\Theta : \mcC \to \mcC$ be a permutation that preserves adjacency and opposition in $\mcC$ and satisfies $\Theta^2 = \mathrm{id}$ and $\Theta(\mcC_+) = \mcC_-$.
The question whether there is an automorphism $\theta$ of $G$ with $\theta^2 = \id$ such that $\theta(B_+)$ is a $G$-conjugate of $B_-$ and such that the automorphism $\wtheta$ from Proposition \ref{prop:BN-quasi-flip-Weyl-auto}(\ref{prop:BN-quasi-flip-Weyl-auto3}) equals $\Theta$, can be answered affirmatively, if $G$ is a characteristic subgroup of $\Aut(\mcC)$; in this case $\Theta$ induces an automorphism of $G$. 
This is for instance the case, if $G$ is the group generated by the root groups $U_\alpha$ of a strictly Moufang twin building.
\end{remark}

In view of the preceding proposition and remark it is natural to  study the involutions of $G$ and of $\mcC$ discussed above simultaneously. In particular, we will not always formally distinguish between $\theta$, $\tilde \theta$ and $\theta_x$, but instead often denote all of these maps by $\theta$.

\begin{defn} \label{def:BN-quasi-flip}
Let $\mcC=((\mcC_+,\delta_+),(\mcC_-,\delta_-),\delta^*)$ be a twin building of type $(W,S)$ and, if $\mcC$ is Moufang, let $G$ be a strongly transitive group of automorphisms of $\mcC$.
Let $\theta$ be either 
\begin{itemize}
\item an automorphism of $G$ or
\item a permutation of $\mcC_+ \cup \mcC_-$.
\end{itemize}
The map $\theta$ is called a \Defn{quasi-flip}, if
\begin{itemize}
\item $\theta^2 = \id$ and $\theta(B_+)$ is a $G$-conjugate of $B_-$, resp.\
\item $\theta^2 = \id$ and $\theta(\mcC_+) = \mcC_-$ and $\theta$ preserves adjacency and opposition of $\mcC$.
\end{itemize}
If $\theta$ flips the distances and preserves the codistance, i.e., 
\begin{itemize}
\item if the induced automorphism $nT\mapsto x^{-1}\theta(n)xT$ from Proposition \ref{prop:BN-quasi-flip-Weyl-auto}(\ref{prop:BN-quasi-flip-Weyl-auto2}) of the Coxeter system $(NT/T,S) \cong (W,S)$ of $G$ is trivial, resp.\
\item if for $\eps \in \{+,-\}$ and for all $x, y \in \mcC_\eps$, $z \in \mcC_{-\eps}$ we have $\delta_\eps(x,y)= \delta_{-\eps}(\theta(x), \theta(y))$ and $\delta^* (x,z) = \delta^* (\theta(x), \theta(z))$,
\end{itemize}
then $\theta$ is called a \Defn{flip}.
\end{defn}

\begin{remark} \label{rmk:spherical-flips}
Recall from Remark \ref{rmk:spherical-flips-1} that there is a canonical one-to-one correspondence between spherical buildings and spherical twin buildings of type $(W,S)$: For a spherical building $(\mcC,\delta)$, the tuple $\mcC_{\mathrm{twin}} := ((\mcC,\delta),(\mcC_-,\delta_-),\delta^*)$ with $\mcC_- := \{ c_- \mid c \in \mcC \}$ and functions $\delta_- : \mathcal{C}_- \times \mcC_- \to W : (c_-,d_-) \mapsto w_S\delta(c,d)w_S$ and $\delta^* : (\mcC \times \mcC_-) \cup (\mcC_- \times \mcC) \to W : (c,d_-) \mapsto \delta(c,d)w_S, (c_-,d) \mapsto w_S\delta(c,d)$
is a twin building, and any twin building with $(\mcC,\delta)$ as a positive half is isomorphic to this twin building.

Two chambers $c,d \in \mcC$ are opposite (i.e., $\delta(c,d)=w_S$) if and only if $c$ and $d_-$ are opposite (i.e., $\delta^*(c,d_-)=\delta(c,d)w_S = 1_W$) if and only if $c_-$ and $d$ are opposite (i.e., $\delta^*(c_-,d)=w_S\delta(c,d) = 1_W$). Therefore, by \cite[Exercise~5.164]{Abramenko/Brown:2008}, the involutive (almost) isometries of spherical buildings correspond to the (quasi-)flips of spherical twin buildings. Using the notation introduced above, this correspondence is given by $(\theta_{\mathrm{isom}} : \mathcal{C} \to \mathcal{C}) \longleftrightarrow (\theta_{\mathrm{flip}} : \mathcal{C}_{\mathrm{twin}} \to \mathcal{C}_{\mathrm{twin}})$ via $\theta_{\mathrm{isom}}(c) = d \Longleftrightarrow \theta_{\mathrm{flip}}(c) = d_- \Longleftrightarrow \theta_{\mathrm{flip}}(c_-) = d$. Note that $\theta_{\mathrm{isom}} = \mathrm{proj}_{\mathcal{C}} \circ {\theta_{\mathrm{flip}}}_{|\mcC}$, where $\mathrm{proj}_{\mathcal{C}}$ denotes the projection from $\mathcal{C}_-$ onto $\mathcal{C}$ within the twin building $\mcC_{\mathrm{twin}}$.
\end{remark}

\begin{example} \label{exmpl:proj-space-form-flip}
Let $V$ be an $n$-dimensional vector space over a field $\F$ with an automorphism $\sigma$ satisfying $\sigma^2 = \mathrm{id}$. Let $(\cdot,\cdot) : V \times V \to \F$ be a non-degenerate $\sigma$-hermitian sesquilinear form. Then the map $U \mapsto \bigcap_{u \in U} \mathrm{ker}(u,\cdot)$ induces an almost isometry on the spherical building described in Example \ref{An3}. In view of Remark \ref{rmk:spherical-flips} this yields a quasi-flip (in fact, a flip) of the corresponding spherical twin building of type $A_{n-1}$.
For $\F = \F_{q^2}$ and $\sigma \neq \mathrm{id}$ this flip of the spherical twin building of type $A_{n-1}$ induces the flip given by transpose-inverse times $\sigma$ of the matrix group $\mathrm{SL}_n(\mathbb{F}_{q^2})$ with respect to a $(\cdot,\cdot)$-orthonormal basis of $V$, as studied in \cite{Bennett/Shpectorov:2004} in order to provide an alternative proof of Phan's local recognition theorem \cite{Phan:1977} of the group $\mathrm{SU}_{n}(\mathbb{F}_{q^2})$.
\end{example}

\begin{remark} \label{abstractpossible}
\begin{enumerate}
\item\label{abstractpossible1} Let $G$ be a connected isotropic reductive linear algebraic group defined over an infinite field $\F$ and let $G(\F)$ denote the subgroup of $\F$-rational points of $G$. Then, by \cite[Proposition~7.2]{Borel/Tits:1973}, any abstract automorphism of $G(\F)$ maps parabolic $\F$-subgroups to parabolic $\F$-subgroups. In particular, any involutive automorphism of $G(\F)$ is a quasi-flip.
\item By \cite[p.\ 132, Corollary]{Steinberg:1968} plus Sylow's theorem any involutive automorphism of a finite group of Lie type is a quasi-flip.
\item Let $\mathcal{D} = (I,A,\Lambda,(c_i)_{i\in I})$ be a non-spherical Kac--Moody root datum, let $\mathcal{F} = (\mathcal{G},(\phi_i)_{i \in I},\eta)$ be the basis of a Tits functor $\mathcal{G}$ of type $\mathcal{D}$, let $\mathbb{F}$ be a field distinct from $\mathbb{F}_2$, $\mathbb{F}_3$, and let $G := \mathcal{G}(\mathbb{F})$ be the corresponding Kac--Moody group. By \cite[Theorem~4.1]{Caprace:2009} any abstract automorphism of $G$ either preserves or interchanges the two conjugacy classes of Borel subgroups of $G$. In particular, any involutive automorphism of $G$ that does not preserve a conjugacy class of Borel subgroups is a quasi-flip.
\end{enumerate}
\end{remark}

\begin{example} \label{Iwahori}
Consider the group $\mathrm{GL}_n(\mathbb{F}[t,t^{-1}])$ discussed in Examples \ref{Antilde} and \ref{Antilde2}. Then the automorphism $\theta : \mathrm{GL}_n(\mathbb{F}[t,t^{-1}]) \to \mathrm{GL}_n(\mathbb{F}[t,t^{-1}])$ that centralizes $\mathrm{GL}_n(\mathbb{F})$ and interchanges $t$ and $t^{-1}$ is a quasi-flip.
\end{example}

\begin{remark} \label{localisom}
Let $\mcC=((\mcC_+,\delta_+),(\mcC_-,\delta_-),\delta^*)$ be a twin building, let $\theta$ be a quasi-flip of $\mcC$, and let $R$ be a spherical residue of $\mcC_\epsilon$, $\epsilon \in \{ +, - \}$. If $R$ and $\theta(R)$ are parallel, then the restriction $\theta_{|R \cup \theta(R)}$ is a quasi-flip of the spherical twin building defined on the pair $(R,\theta(R))$ by the canonical (co)restrictions of the distances and the codistance. The corresponding involutive almost isometry of $R$ as described in Remark \ref{rmk:spherical-flips} is given by the product $\mathrm{proj}_{R} \circ \theta_{|R}$. 
\end{remark}

It will turn out to be useful to know how far a quasi-flip of a twin building moves a chamber. This notion is made precise in the following definition:

\begin{defn} \label{theta-cod}
Let $\theta$ be a quasi-flip of a twin building $\mathcal{C}$.
For a chamber $c \in \mcC$, the value $\dt(c) := \delta^*(c,\theta(c)) \in W$ is called the \Defn{\thcod} of $c$. Moreover, $\lt(c):=l(\dt(c))$ is called the \Defn{numerical \thcod}.
If there exists a $c \in \mcC$ with $\dt(c) = 1_W \in W$, then the quasi-flip $\theta$ is called a \Defn{proper quasi-flip} or a \Defn{Phan involution}.

A quasi-flip $\theta$ of a group with a $BN$-pair is called \Defn{proper}, if the quasi-flip $\wtheta$ from Proposition \ref{prop:BN-quasi-flip-Weyl-auto}(\ref{prop:BN-quasi-flip-Weyl-auto3}) is proper.
\end{defn}

For proper quasi-flips the set of chambers opposite their image under $\theta$ has a rich and interesting geometric structure. In order to also capture non-proper quasi-flips, we propose the following definition.

\begin{defn}
Let $\theta$ be a quasi-flip of a Moufang twin building $\mathcal{C}$.
A residue $R$ of $\mcC$ is called a \Defn{Phan residue}, if $R$ is opposite $\theta(R)$, i.e., for each chamber of $R$ there exists an opposite chamber in $\theta(R)$, and vice versa. 
A Phan residue consisting of a single chamber is called a \Defn{Phan chamber}.
\end{defn}

Another important class of quasi-flips has first appeared in \cite{Devillers/Muehlherr:2007}:

\begin{defn}[Definition 6.2 of \cite{Devillers/Muehlherr:2007}] \label{non-strong}
Let $\theta$ be a building quasi-flip of a twin building $\mcC$. For any spherical residue $R$, define the set $\proj_R(\theta):=\{ c\in R \mid \proj_R(\theta(c))=c\}$,
where $\proj_R$ denotes the projection onto $R$.
If for all panels $P$ of $\mcC$ one has $\proj_P(\theta)\neq P$, one calls $\theta$ a \Defn{strong quasi-flip}.
\end{defn}

\begin{remark}
A prototypical example of a strong flip is given in \ref{exmpl:proj-space-form-flip}, in case $\sigma \neq \mathrm{id}$; see Proposition \ref{lem:semi-linear-strong} below.
\end{remark}

\section{$\theta$-twisted involutions and minimal Phan residues} \label{sec:steep-descent} \label{sec:strong-flips}
%  ___         _   _          
% / __| ___ __| |_(_)___ _ _  
% \__ \/ -_) _|  _| / _ \ ' \ 
% |___/\___\__|\__|_\___/_||_|

Let $\theta$ be a quasi-flip of a twin building $\mathcal{C}$, let $c \in \mcC$, and let $\dt(c) = \delta^*(c,\theta(c))= w \in W$. Then $w^{-1} =  \delta^*(\theta(c),c) = \theta(w)$.
Such Weyl group elements are called twisted involutions in \cite[Section 3]{Springer:1984} and \cite[Section 7]{Helminck/Wang:1993}; see also \cite{Hultman:2007}, \cite{Hultman:2008}.

\begin{defn} \label{itheta}
Let $(W,S)$ be a Coxeter system and let $\theta$ be an automorphism of $(W,S)$ of order at most two.
A \Defn{$\theta$-twisted involution} of $W$ is an element $w\in W$ satisfying $\theta(w)=w^{-1}$. The set of $\theta$-twisted involutions is denoted by $\Inv^\theta(W)$.
\end{defn}
Of course, $\Inv^{\Id}(W)=\Inv(W)$ is just the set of elements of $W$ that square to the identity. 

\begin{lemma} \label{lem:cox-inv3}
Let $w\in \Inv^\theta(W)$ be a $\theta$-twisted involution and let $s\in S$. Then $l(sw)=l(w\theta(s))$. Moreover, if $l(sw\theta(s))=l(w)$, then $sw\theta(s)=w$.
\end{lemma}
\begin{proof}
Since $\theta$ is an automorphism of $(W,S)$, any $w\in W$ satisfies $l(w)=l(\theta(w))$.
Hence $l(sw) = l(w^{-1} s^{-1}) = l(\theta(w) s) = l(w\theta(s))$.

The second statement is a consequence of \cite[Lemma 3.2]{Springer:1984}:
If $l(sw)<l(w)$, then one may write $w=s_1\cdots s_h$ with $s_i\in S$,
$l(w)=h$, and $s_1=s$. Since $w$ is a $\theta$-twisted involution, one obtains
$w=\theta(s_h)\cdots\theta(s_1)$. The exchange condition implies
$sw=\theta(s_h)\cdots\widehat{\theta(s_i)}\cdots\theta(s_1)$ for some $i$. If
$i>1$, then $l(sw\theta(s))<l(w)$, a contradiction to the hypothesis
$l(sw\theta(s))=l(w)$. Hence $i=1$ and $sw\theta(s)=w$. If $l(sw)>l(w)$, then,
for $w=s_1\cdots s_h$ with $s_i\in S$ and $l(w)=h$, one has
$sw=s\theta(s_h)\cdots\theta(s_1)$, since $w$ is a $\theta$-twisted
involution. By hypothesis, $l(sw\theta(s))=l(w)<l(sw)$, so that the exchange
condition implies that $sw\theta(s)$ equals $\theta(s_h)\cdots\theta(s_1)$ or
$s\theta(s_h)\cdots\widehat{\theta(s_i)}\cdots\theta(s_1)$ for some $i$. Since
$sw\theta(s)$ is a $\theta$-twisted involution, in the latter case one obtains $sw\theta(s) = s_1 \cdots \widehat{s_i} \cdots s_h\theta(s)$, whence $l(sw) < l(sw\theta(s))$, again a contradiction. Therefore $sw\theta(s) = \theta(s_h)\cdots\theta(s_1) = w$.
\end{proof}

\begin{prop}[Proposition 3.3(a) of \cite{Springer:1984}]  \label{prop:cox-twisted-inv}
Let $w\in \Inv^\theta(W)$ be a $\theta$-twisted involution. Then there exist a $\theta$-stable spherical subset $I$ of $S$ and $s_1,\ldots,s_h\in S$ such that
$w = s_1\cdots s_h w_I \theta(s_h)\cdots\theta(s_1)$
with $l(w)=l(w_I)+2h$, where $w_I$ denotes the longest word of the spherical Coxeter system $(\langle I \rangle, I)$.
\end{prop}
\begin{proof}
We prove the claim by induction on $l(w)$ based on the trivial case $w=1_W$. Let $l(w)>0$ and assume that the claim is true for all $\theta$-twisted involutions $w'$ with $l(w) > l(w')$. If there exists $s\in S$ with $l(sw\theta(s))=l(w)-2$, then by induction there is nothing to show.
By Lemma \ref{lem:cox-inv3} it therefore remains to deal with the situation that each $s\in S$ with $l(sw)<l(w)$ satisfies $sw\theta(s)=w$.
By \cite[Proposition~2.17 and Corollary~2.18]{Abramenko/Brown:2008} the set $I:=\{s\in S\mid l(sw)<l(w)\}$ is spherical and each reduced $I$-word can occur as an initial subword of a reduced decomposition of $w$; in particular $l(w_Iw)=l(w)-(w_I)$.
Hence, if there exists $s \in S$ such that $l(w_I ws)<l(w_I w)$, then $l(ws)<l(w)$. In this case Lemma \ref{lem:cox-inv3} implies $l(\theta(s)w)<l(w)$, so that $\theta(s)\in I$. Therefore $$l(\theta(s)ws)=l(\theta(s)w_Iw_Iws)\leq l(\theta(s)w_I)+l(w_Iws)=l(w_I)-1+l(w_Iw)-1=l(w)-2,$$ a contradiction to $\theta(s)ws=w$.
This implies that $s \in S$ with $l(w_I ws)<l(w_I w)$ cannot exist, so that $w_I = w$. Finally, the observation $\theta(w) = w^{-1} = w$ implies that $\theta(I) = I$.
\end{proof}

\begin{remark}
In case $\theta = \mathrm{id}$ the above results give information on the structure of conjugacy classes of involutions in Coxeter groups; see also \cite{Richardson:1982c}.
\end{remark}

We can now draw some conclusions for $\theta$-codistances. 

\begin{lemma} \label{lem:equal-thcod}
If two adjacent chambers have the same numerical \thcod, then they have the same \thcod.
\end{lemma}
\begin{proof}
Let $c$, $c'$ be $s$-adjacent chambers, which implies $\dt(c') \in \{ \dt(c), s\dt(c), \dt(c)\theta(s), s\dt(c)\theta(s) \}$. From $\lt(c) = \lt(c')$ we conclude $\dt(c') = \dt(c)$ or $\dt(c') = s\dt(c)\theta(s)$. In the latter case necessarily $l(s\dt(c)\theta(s))=l(\dt(c))$, so that Lemma \ref{lem:cox-inv3} yields $\dt(c') = s\dt(c)\theta(s) = \dt(c)$.
\end{proof}

The following lemma is useful when trying to determine whether a panel $P$ is parallel to its image $\theta(P)$.

\begin{lemma} \label{lem:flip-steep-descent}
Let $c$ be a chamber with \thcod $w:=\dt(c)$ and let $s \in S$. 
\begin{enumerate}
\item \label{lem:flip-steep-descent1} If $l(sw\theta(s))=l(w)-2$, then each
  chamber in $P_s(c)\setminus\{c\}$ has \thcod $sw\theta(s)$. Moreover, $\mathrm{proj}_{P_s(c)}(\theta(P_s(c))) = \{ c \}$ and $\mathrm{proj}_{\theta(P_s(c))}(P_s(c)) = \{ \theta(c) \}$.
\item \label{lem:flip-steep-descent2} If $l(sw\theta(s))=l(w)$, then the panels $P_s(c)$ and $\theta(P_s(c))$ are parallel. Moreover, $\mathrm{proj}_{P_s(c)}(\theta(c)) = c$ if and only if $l(sw) < l(w)$.
\item \label{lem:flip-steep-descent3} If $l(sw\theta(s))=l(w)+2$, then there exists a unique chamber $d \in P_s(c)$ with $\dt(d) = sw\theta(s)$. Moreover, $\mathrm{proj}_{P_s(c)}(\theta(P_s(c))) = \{ d \}$ and $\mathrm{proj}_{\theta(P_s(c))}(P_s(c)) = \{ \theta(d) \}$.
\end{enumerate}
In particular, $P_s(c)$ is parallel to $\theta(P_s(c))$ if and only if $l(w)=l(sw\theta(s))$ if and only if $w=sw\theta(s)$. 
\end{lemma}

\begin{proof}
(\ref{lem:flip-steep-descent1}) Let $d\in P_s(c)\setminus\{c\}$. Since
$l(sw\theta(s))=l(w)-2$, one has $l(sw)=l(w\theta(s))=l(w)-1$, and the twin building axioms first imply $\delta^*(c,\theta(d))=w\theta(s)$ and then $\delta^*(d,\theta(d))=sw\theta(s)$. The second claim follows from $l(\delta^*(c,\theta(d)))=l(w\theta(s)) >  l(sw\theta(s)) = l(\delta^*(d,\theta(d)))$.

(\ref{lem:flip-steep-descent2})
Assume by way of contradiction that the panels $P_s(c)$ and $\theta(P_s(c))$
are not parallel. If $\mathrm{proj}_{P_s(c)}(\theta(P_s(c))) = \{ c \}$, then
for each $d \in P_s(c)\setminus\{c\}$ one has $l(\delta^*(d,\theta(c))) < l(\delta^*(c,\theta(c)))$ and, thus, $\delta^*(d,\theta(c)) = sw$. Application of $\theta$ implies $\mathrm{proj}_{\theta(P_s(c))}(P_s(c)) = \{ \theta(c) \}$, so that $l(\delta^*(d,\theta(d))) < l(\delta^*(d,\theta(c)))$ and, necessarily, $\delta^*(d,\theta(d))) = sw\theta(s)$, a contradiction to $l(sw\theta(s))=l(w)$. If $\mathrm{proj}_{P_s(c)}(\theta(P_s(c))) = \{ x \} \neq \{ c \}$, then $l(\delta^*(x,\theta(c))) > l(\delta^*(c,\theta(c)))$ and, thus, $\delta^*(x,\theta(c)) = sw$. Application of $\theta$ implies $\mathrm{proj}_{\theta(P_s(c))}(P_s(c)) = \{ \theta(x) \} \neq \{ \theta(c) \}$ and therefore $l(\delta^*(x,\theta(x))) > l(\delta^*(x,\theta(c)))$, whence $\delta^*(x,\theta(x))) = sw\theta(s)$, another contradiction.
The second claim follows by similar arguments. 

(\ref{lem:flip-steep-descent3})
Let $d := \mathrm{proj}_{P_s(c)}(\theta(c))$. Since $l(sw) > l(w)$, necessarily $d \neq c$. Since $l(sw\theta(s)) > l(sw)$, moreover $\mathrm{proj}_{\theta(P_s(c))}(d) \neq \theta(c)$. Hence $\mathrm{proj}_{\theta(P_s(c))} : P_s(c) \to \theta(P_s(c))$ is not surjective, so that $\theta(d) = \mathrm{proj}_{\theta(P_s(c))}(c)$ implies $\theta(d) = \mathrm{proj}_{\theta(P_s(c))}(d)$. Therefore the claim follows from item (\ref{lem:flip-steep-descent1}).

\medskip Items (\ref{lem:flip-steep-descent1})--(\ref{lem:flip-steep-descent3}) imply the first equivalence of the final statement. The second equivalence follows from Lemma \ref{lem:cox-inv3}.
\end{proof}

\begin{lemma} \label{lem:strong-flip-descent}
 Let $\theta$ be a strong quasi-flip and let $c$ be a chamber with \thcod
 $w$. If there exists $s\in S$ with $l(sw)<l(w)$, then there is a chamber $d$ with $\lt(d) < \lt(c)$ that is $s$-adjacent to $c$. In particular, strong quasi-flips are proper.
\end{lemma}
\begin{proof}
For $w=1_W$ there is nothing to show. Otherwise, there exists $s\in S$ such
that $l(sw)<l(w)$. Lemma
\ref{lem:flip-steep-descent}(\ref{lem:flip-steep-descent1})(\ref{lem:flip-steep-descent2})
implies $\proj_{P_s(c)}(\theta(c))=c$ for 
the $s$-panel $P_s(c)$ containing $c$. 
As $\theta$ is strong, there exists a chamber $d \in P_s(c)$ with $\proj_{P_s(c)}(\theta(d))\neq d$. Hence $\lt(d) < \lt(c)$, so that by induction there exists a Phan chamber.
\end{proof}

The following is a $\theta$-twisted variation of a well-known result.

\begin{lemma}  \label{lem:pretty-cool-descent}
Let $r\in W$ be a $\theta$-twisted involution and let $w\in W$ such that $l(w^{-1}r\theta(w))=l(r)-2l(w)$. 
\begin{enumerate}
\item\label{lem:pretty-cool-descent1} If $c,d\in\mcC_\eps$ with $\dt(c)=r$ and $\delta_\eps(c,d) = w$, then $\dt(d)=w^{-1}r\theta(w)$.
\item\label{lem:pretty-cool-descent2} If $d\in\mcC_\eps$ with $\dt(d)=w^{-1}r\theta(w)$, then there exists a unique chamber $c$ with $\delta_\eps(c,d) = w$ and $\dt(c)=r$.
\end{enumerate}
In either case, the convex hull of $d$ and $\theta(d)$ contains $c$ and $\theta(c)$.
\end{lemma}
\begin{proof}
One may write $w=s_1\cdots s_h$ with $s_i\in S$ and $l(w)=h$. 

(\ref{lem:pretty-cool-descent1})
Let $(c=c_0\sim_{s_1} c_1 \sim_{s_2} \cdots \sim_{s_h} c_n=d)$ be a minimal gallery from $c$ to $d$. By hypothesis $l(s_1r)=l(r)-1=l(r\theta(s_1))$, whence Lemma \ref{lem:flip-steep-descent}(\ref{lem:flip-steep-descent1}) implies $\dt(c_1)=s_1 r \theta(s_1)$. An iteration of this argument yields $\dt(d)=(s_h\cdots s_1) r  \theta(s_1 \cdots s_h) =  w^{-1}r\theta(w)$.

(\ref{lem:pretty-cool-descent2}) We prove this claim by an induction on $h$. Since for $h=0$ there is nothing to show, let $h>0$ and assume the claim holds for each $w' \in W$ with $l(w) > l(w')$. By Lemma \ref{lem:flip-steep-descent}(\ref{lem:flip-steep-descent3}) there exists a unique chamber $d'\in P_{s_h}(d)$ with $\dt(d') = s_hw^{-1}r\theta(w)\theta(s_h)= (s_{h-1}\cdots s_1) r \theta(s_1 \cdots s_{h-1})$, and the claim follows from the induction hypothesis.

\medskip
The final statement is an immediate consequence of (\ref{lem:pretty-cool-descent2}).
\end{proof}

We now apply our findings to study Phan residues which are minimal with respect to inclusion.

\begin{prop}\label{new}\label{lem:phanres-thcod-WI}\label{lem:phanres-intersect}\label{lem:min-phanres-sph}\label{lem:min-phanres-thcod-const}
\begin{enumerate}
\item\label{new1}If $c \in \mcC_\eps$ and $I \subseteq S$ with $\dt(c) \in W_I$, then for each $I \subseteq J \subseteq S$ the residue $R_J(c)$ is a Phan residue.
\item\label{new2}\label{lem:phanres-thcod-WI2}\label{lem:phanres-intersect2} If $R$ is a Phan residue of type $I$, then the image of the restriction of the \thcod to $R$ is contained in $W_I$. Moreover, if $R_1$ and $R_2$ are Phan residues with non-empty intersection, then $R_1\cap R_2$ is also a Phan residue.
\item\label{new3}\label{lem:min-phanres-thcod-const3}\label{lem:min-phanres-sph3}
  If $R$ is a minimal Phan residue of type $I$, then $I$ is spherical and the
  restriction of the \thcod to $R$ is constant with value $w_I$, the longest element of $W_I$.
\end{enumerate}
\end{prop}

\begin{proof}
(\ref{new1}) If $d$ is any chamber satisfying $\delta_\eps(c,d) = \dt(c) \in W_I$, then it is contained in $R_J(c)$ and we deduce from \cite[Lemma 5.140(1)]{Abramenko/Brown:2008} that $$\delta^*(c,\theta(d))
 \leq \delta^*(c,\theta(c)) \delta_{-\eps}(\theta(c),\theta(d))
 = \delta^*(c,\theta(c)) \theta(\delta_\eps(c,d))=\delta^*(c,\theta(c)) \theta(\delta^*(c,\theta(c))) = 1_W.$$ Application of $\theta$ yields $\delta^*(\theta(c),d)= 1_W$ and hence $R_J(c)$ is a Phan residue.

(\ref{new2}) Let $c\in R \subseteq \mcC_\eps$. Since $R$ is a Phan residue, by definition there exists $\theta(d)\in \theta(R)$ opposite $c$. Another application of \cite[Lemma 5.140(1)]{Abramenko/Brown:2008} yields $$\dt(c)=\delta^*(c,\theta(c))\leq \delta^*(c,\theta(d))\delta_{-\eps}(\theta(d),\theta(c)) = \delta_{-\eps}(\theta(d),\theta(c)) \in W_I.$$ 
Therefore, if $c\in R_1\cap R_2 \subseteq \mcC_\eps$ and $I_1$ and $I_2$ are the types of $R_1$ and $R_2$, respectively, then $\dt(c)\in W_{I_1\cap I_2}$. Hence $R_1 \cap R_2 = R_{I_1 \cap I_2}(c)$ is a Phan residue by (\ref{new1}).

(\ref{new3}) Let $Q \subseteq \mcC_\eps$ be a Phan residue of type $J$ and let $c \in Q$. By Proposition \ref{prop:cox-twisted-inv} and Lemma \ref{lem:pretty-cool-descent} there exist a spherical subset $I$ of $S$, an element $w'\in W$ with $w'\leq \dt(c)$ in the Bruhat order, and a chamber $c'\in\mcC_\eps$ satisfying $\delta_\eps(c,c')=w'$ and $\dt(c')=w_I$.
Since $\dt(c)\in W_J$ by (\ref{lem:phanres-thcod-WI2}), the element $w'<\dt(c)$ is contained in $W_J$ as well. Hence $c'\in Q$.
The residue $R_I(c') \subseteq Q$ is spherical and by (\ref{new1}) it is a Phan residue, because $\dt(c')=w_I \in W_I$. Hence each Phan residue contains a spherical Phan residue, so that minimal Phan residues are spherical.

It remains to show that the $\theta$-codistance is constant on $R_I(c')$, if $R_I(c')$ is a minimal Phan residue. For $d\in R_I(c')$, Proposition \ref{prop:cox-twisted-inv} and Lemma \ref{lem:pretty-cool-descent} imply the existence of a spherical subset $I'$ of $I$ and a chamber $d'\in R_I(c')$ such that $\dt(d')=w_{I'} \leq \dt(d) \leq w_I$. By (\ref{new1}) the residue $R_{I'}(d') \subseteq R_I(c')$ is a Phan residue. Minimality of $R_I(c')$ implies $R_{I'}(d') = R_I(c')$, whence $I' = I$. Therefore the claim follows from $w_{I} \leq \dt(d) \leq w_I$.
\end{proof}

%%%%%%%%%%%%%%%%%%%%%%%%%%%%%%%%%%%%%%%%%%%%%%%%%%%%%%%%%%%%%%%%%%%%%%%%%%
%%%%%%%%%%%%%%%%%%%%%%%%%%%%%%%%%%%%%%%%%%%%%%%%%%%%%%%%%%%%%%%%%%%%%%%%%%

\section{Structure of flip-flop systems} \label{chap:geom} \label{sec:geom-homo-inhConn} \label{sec:geom-main-thms} \label{sec:flip-flop-residual-conn}

Let $\mcC=((\mcC_+,\delta_+),(\mcC_-,\delta_-),\delta^*)$ be a twin building of type $(W,S)$ and let $\theta$ be a quasi-flip of $\mcC$. One of our goals in this section is to determine under which assumptions on $\mcC$ and $\theta$ the minimal Phan residues of $\mcC$ are of some constant type $K \subset S$; if that happens, then the pair $(\mcC, \theta)$ is called \Defn{homogeneous} or, more precisely, \Defn{$K$-homogeneous}. 
By Proposition \ref{lem:min-phanres-sph}(\ref{lem:min-phanres-sph3}) that type $K$ is necessarily spherical. Moreover, the numerical $\theta$-codistance function reaches a local minimum at the chambers of a minimal Phan residue. 

The second main object of study in this section, the flip-flop system, consists of all chambers on which the numerical $\theta$-codistance function is at a global minimum in the following sense:

\begin{defn} \label{dfn-flipflop}
Let $\theta$ be a quasi-flip of a twin building $\mcC=((\mcC_+,\delta_+),(\mcC_-,\delta_-),\delta^*)$ and let $R$ be a residue of $\mcC_\eps$, $\eps \in \{ +, - \}$. 
The \Defn{minimal numerical \thcod} of $R$ is defined as $\min_{c\in R}\; \lt(c) = \min_{c\in R}\; l(c,\theta(c)). $
 
Furthermore, the \Defn{induced flip-flop system} $\Rt$ on $R$ associated to $\theta$ is the sub-chamber system
  $\Rt := \{c \in R \mid \lt(c) = \min_{d\in R}\; \lt(d) \}$ of $\mcC_\eps$ with
the induced adjacency relations. For $R=\mcC_\eps$ the chamber system $\Ct := \mcC_\eps^\theta$ is called the \Defn{flip-flop system} on $\mcC$ associated to $\theta$.
\end{defn}

\begin{example} \label{BSflipflop}
Consider the flip $\theta$ given by transpose-inverse times field involution of the matrix group $\mathrm{SL}_n(\mathbb{F}_{q^2})$ discussed in Example \ref{exmpl:proj-space-form-flip} and let $(\mathcal{C},(\sim_{s_i})_i)$ be the building of type $A_{n-1}$ described in Example \ref{An3}, modelled on an $n$-dimensional $\mathbb{F}_{q^2}$-vector space $V$, endowed with the non-degenerate hermitian form $(\cdot,\cdot)$ induced by $\theta$. Then the set of chambers of the flip-flop system $\mathcal{C}^\theta$ equals $\{(V_1,...,V_{n-1}) \in \mathcal{C} \mid \mbox{$V_i$ is $(\cdot,\cdot)$-nondegenerate for each $1 \leq i \leq n-1$ } \}$.
\end{example}

\begin{remark}
The connectedness properties of the set $\Rt$ introduced in the preceding definition play a crucial role in the main result of \cite[Section 6]{Devillers/Muehlherr:2007}, where it is called $A_\theta(R)$.
\end{remark}

\begin{example}
By \cite{Bennett/Shpectorov:2004} the flip-flop system $\mathcal{C}^\theta$ described in Example \ref{BSflipflop} is connected, if $n \geq 3$ and $q \geq 3$, and simply connected, if $n \geq 4$ and $q \geq 4$.
\end{example}

\begin{remark} \label{lem:K-hom-then-Ct-union}
Note that Proposition \ref{prop:cox-twisted-inv} and Lemma \ref{lem:pretty-cool-descent} imply that each element of the flip-flop system $\Ct$ is contained in a minimal Phan residue. Therefore $\Ct \subseteq \mcC_\eps$ is a subset of the union of all minimal Phan residues of $\mcC_\eps$. Moreover, if $\theta$ is a homogeneous quasi-flip, then $\Ct\subseteq \mcC_\eps$ equals the union of all minimal Phan residues of $\mcC_\eps$ by Proposition \ref{lem:min-phanres-thcod-const}(\ref{lem:min-phanres-thcod-const3}).
\end{remark}

The following property allows us to measure from where the flip-flop system can be reached by using the strategy of local descent.

\begin{defn}
Let $R$ be a residue of $\mcC$. 
For $X \subseteq R$ the residue $R$ admits \Defn{direct descent from $X$ into $\Rt$}, if for each chamber $c \in X$ there exists a gallery in $R$ from $c$ to some chamber in $\Rt$ with the property that $\lt$ strictly decreases along the gallery.
If $R$ admits {direct descent from $R$ into $\Rt$}, then one also says that  $R$ admits \Defn{direct descent into $\Rt$}.
\end{defn}

\begin{lem} \label{lem:rank2-bypass}
Let $\theta$ be a quasi-flip of a twin building $\mcC$, let $R$ be a rank two residue that admits direct descent into a connected $\Rt$, and let $(c_0,c_1,c_2) \subseteq R$ be a gallery satisfying $\lt(c_0) < \lt(c_1)$ and $\lt(c_1) \geq \lt(c_2)$. Then there exists a gallery $\gamma \subseteq R$ from $c_0$ to $c_2$ such that $\lt(c_1) > \lt(d)$ for all $d \in \gamma\setminus\{c_2\}$ (see Figure \ref{fig:peak-plateau}).
\end{lem}

\begin{figure}[h]
\centering
\includegraphics{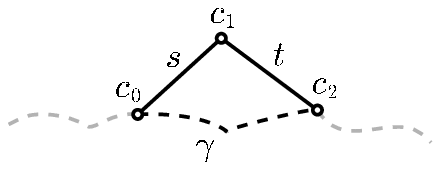}
\quad
\includegraphics{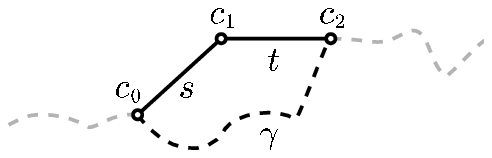}
\caption{A peak and a short plateau with bypasses in the \thcod of a gallery. Higher numerical \thcod is reflected by chambers being depicted farther upwards.}
\label{fig:peak-plateau}
\end{figure}

\begin{proof}
By hypothesis there exists a gallery $\gamma_0 \subseteq R$ from $c_0$ to some chamber $c'_0 \in \Rt$ such that $\lt$ strictly decreases along that gallery. In particular, $\lt(d) \leq \lt(c_0) < \lt(c_1)$ for each $d \in \gamma_0$. Likewise we find a gallery $\gamma_2 \subseteq R$ from $c_2$ to some chamber $c'_2\in\Rt$ on which $\lt$ strictly decreases. Hence, $\lt(d) < \lt(c_2) \leq \lt(c_1)$ for all $d \in \gamma_2 \backslash \{ c_2 \}$. 
Since $\Rt$ is connected by hypothesis, there exists a gallery $\gamma_1$ in $\Rt$ from $c'_0$ to $c'_2$. As $c_1 \not\in \Rt$ by hypothesis, we conclude that $\gamma=\gamma_0 \gamma_1 \gamma_2^{-1}$ is a gallery as required.
\end{proof}

Repeated application of Lemma \ref{lem:rank2-bypass} yields the following structure result. (Recall the concept of inherited connectedness introduced in the part on chamber systems of Section \ref{basics}.) 

\begin{prop} \label{prop:good-rank2-implies-all-good} \label{prop:good-rank2-implies-hom-inhConn}
Let $\theta$ be a quasi-flip of a twin building $\mcC$ such that each rank two residue $R$ admits direct descent into a connected $\Rt$. Then the following hold:
\begin{enumerate}
\item\label{prop:good-rank2-implies-all-good1} Each residue $Q$ of $\mcC_\eps$ admits direct descent into $\Qt$. Moreover, $\Qt$ inherits connectedness from $Q$ and, in particular, is connected.
\item\label{prop:good-rank2-implies-all-good2} The pair $(\mcC, \theta)$ is homogeneous and $\Ct$ inherits connectedness from $\mcC_\eps$.
\end{enumerate}
\end{prop}
\begin{proof}
(\ref{prop:good-rank2-implies-all-good1}) Let $c\in Q$, let $d\in \Qt$, and let $\gamma=(c=c_0,c_1,\ldots,c_n=d)\subseteq Q$ be a gallery from $c$ to $d$. 
Define $m:=\mathrm{max}\{ \lt(x) \mid x \in \gamma\}$ and $X_\gamma:= \{ x \in \gamma \mid \lt(x) = m\}$.

If $c = c_0 \not\in X_\gamma$, let $i$ be minimal with the property that $c_i \in X_\gamma$ and consider the gallery $(c_{i-1},c_i,c_{i+1})$. Lemma \ref{lem:rank2-bypass} implies the existence of a gallery $\widehat{\gamma}\subseteq Q$ from $c_{i-1}$ to $c_{i+1}$
such that $\widehat{\gamma} \cap X_\gamma \subseteq \{ c_{i+1} \}$. Substituting $\widehat{\gamma}$ for the subgallery $(c_{i-1},c_i,c_{i+1})$ in $\gamma$ hence decreases the cardinality of $X_\gamma$. Therefore induction on the cardinality of $X_\gamma$ yields a gallery $\gamma'\subseteq Q$ from $c$ to $d$ satisfying $\mathrm{max}\{ \lt(x) \mid x \in \gamma' \} < m$. 
By an induction on $m$ we can transform $\gamma'$ into a gallery $\gamma''\subseteq Q$ from $c$ to $d$ with $c=c_0 \in X_{\gamma''}$. 

Therefore we may assume that $c\in X_\gamma$, which implies that all chambers in $\gamma$ have numerical \thcod at most $\lt(c)$. In particular, if $c \in \Qt$, then $\gamma\subset \Qt$, so that we have shown that $\Qt$ inherits connectedness from $Q$.
In order to establish direct descent we now proceed by an induction on $\lt(c)$, for which the case $c \in \Qt$, which is equivalent to $\lt(c) = \lt(d)$, serves as a basis.
Let $\lt(c) > \lt(d)$ and assume that, for $X := \{ x \in Q \mid \lt(x) < \lt(c) \}$, the residue $Q$ admits direct descent from $X$ into $\Qt$.
As $\lt(c)>\lt(d)$, the chamber $c_n=d$ is not contained in $X_\gamma$, so that there exists a minimal $i$ such that $c_i\in X_\gamma$ and $c_{i+1}\notin X_\gamma$. 
If $i>0$, we can apply Lemma \ref{lem:rank2-bypass} to the gallery $(c_{i+1},c_{i},c_{i-1})$ to obtain a gallery $\widehat{\gamma}$ from $c_{i+1}$ to $c_{i-1}$ which satisfies $\lt(x) < \lt(c_i) = \lt(c_0)$ for all $x \in \widehat{\gamma} \backslash \{ c_{i-1} \}$. Therefore induction on $i$ yields a gallery $\gamma'$ in which $\lt(c_0) > \lt(c_1)$. 
Since $c_1 \in X$, statement (\ref{prop:good-rank2-implies-all-good1}) follows.

(\ref{prop:good-rank2-implies-all-good2}) By (\ref{prop:good-rank2-implies-all-good1}) the chamber system $\Ct$ inherits connectedness from $\mcC_\eps$. Hence $\Ct$ is connected and so, by Lemma \ref{lem:equal-thcod}, the \thcod on $\Ct$ is constant and equal to some $w\in W$.
  By Proposition~\ref{prop:cox-twisted-inv}, there exists a spherical $K\subseteq S$ such that $w=w_K$, the longest word of $(\langle K \rangle, K)$.

If $R$ is an arbitrary minimal Phan residue of type $I$, then,  by Proposition \ref{lem:min-phanres-thcod-const}(\ref{lem:min-phanres-thcod-const3}), $I$ is spherical and the \thcod on $R$ is constant and equal to $w_I$. Since $\mcC_\eps$ admits direct descent from $R$ to $\Ct$ by (\ref{prop:good-rank2-implies-all-good1}) and since for each $w' \in W$ the fact $w' \leq w_I$ implies $w' \in W_I$, each path descending from any $c \in R$ into $\Ct$ is fully contained in $R$. Therefore $R\subset \Ct$ and $I=K$.
\end{proof}

This proposition turns out to be a very useful tool thanks to the following result on involutions of Moufang polygons. For the definition of unique $2$-divisibility we refer to Definition~\ref{2d}.

\begin{theorem}[Horn, Van Maldeghem] \label{hvm}
Let $(W,S)$ be a Coxeter system of rank two, let $\Phi$ be its root system, let
$\mathcal{C}$ be a Moufang polygon of type $(W,S)$ with uniquely $2$-divisible root groups $\{U_\alpha\}_{\alpha \in \Phi}$, and let $\theta$ be an involutive permutation of $\mathcal{C}$ that preserves adjacency and opposition.

If
\begin{itemize}
\item $\mathcal{C}$ is split and for each $\alpha \in \Phi$ one has $|U_\alpha| \geq 5$, {or}
\item $\mathcal{C}$ is finite and for each $\alpha \in \Phi$ one has $|U_\alpha| \geq 5$,
\end{itemize}
then $\mathcal{C}$ admits direct descent into a connected $\mathcal{C}^\theta$.
\end{theorem}

In the case that $\mathcal{C}$ is a projective plane or a classical quadrangle this theorem follows from \cite[Theorem~4.7.1]{Horn:2008b}. The general situation has been proved by Hendrik Van Maldeghem and the second author, and will become publicly available in the near future. 
 
\begin{example}
Let  $V=\F^{2n}$ endowed with an alternating bilinear form $(\cdot,\cdot)$. Then, as in Example~\ref{exmpl:proj-space-form-flip}, the map $\mathbb{P}(V) \to \mathbb{P}(V) : U \mapsto \bigcap_{u \in U} \mathrm{ker}(u,\cdot)$ is a flip. However, unlike the situation of Example~\ref{exmpl:proj-space-form-flip}, this flip is $\{ s_1, s_3, \ldots, s_{2n-1}\}$-homogenous rather than $\emptyset$-homogeneous.
\end{example}

$K$-homogeneity of $\Ct$ allows one to factor out minimal Phan residues by passing to its $K$-residue chamber system (cf.\ the part on chamber systems in Section \ref{basics}).

\begin{prop} \label{prop:hom-inhConn-imply-resConn} \label{thm:nice-flips-are-geometric}
Let $\mcC$ be a twin building and let $\theta$ be a quasi-flip of $\mcC$.
If $(\mcC, \theta)$ is $K$-homogeneous and if $\Ct$ inherits connectedness from $\mcC_\eps$,
then the $K$-residue chamber system $\Ct_K$ is residually connected.
\end{prop}
\begin{proof}
Let $I=S\setminus{K}$ denote the type set of $\Ct_K$.
Let $J \subseteq I$ and, for each $j \in J$, let $R_j$ be a residue of $\Ct_K$ of type $I \backslash \{ j \}$, i.e., a residue of $\Ct$ of type $S \backslash \{ j \}$. We have to prove that $R_J := \bigcap_{j \in J} R_j$ is non-empty and connected. For each $j$ in $J$ define $\overline{R}_j$ to be the unique $(S\setminus\{j\})$-residue of $\mcC_\eps$ which contains $R_j$. Each of the $\overline{R}_j$ is a Phan residue by Proposition \ref{new}(\ref{new1}), because it contains a $K$-residue in $\mcC_\eps$ of a chamber in $\Ct$. 
Since $\mcC_\eps$ is residually connected, the intersection $\overline{R}_J:=\bigcap_{j\in J} \overline{R}_j$ is non-empty and connected. By Proposition \ref{lem:phanres-intersect}(\ref{lem:phanres-intersect2}) this intersection $\overline{R}_J$ is a Phan residue, which therefore contains a minimal Phan residue of type $K$. Thus, by Remark \ref{lem:K-hom-then-Ct-union} the intersection $\overline{R}_J \cap \Ct = (\bigcap_{j \in J} \overline{R}_j) \cap \Ct = \bigcap_{j \in J} R_j = R_J$ is non-empty.
For $c,d\in R_J\subseteq\overline{R}_J$ there exists an $(I\setminus J)$-gallery in $\mcC_\eps$ from $c$ to $d$. As $\Ct$ inherits connectedness from $\mcC_\eps$, there also exists an $(I\setminus J)$-gallery in $\Ct$ from $c$ to $d$, whence $R_J$ is connected.
\end{proof}

If $K=\emptyset$, then Proposition \ref{prop:hom-inhConn-imply-resConn} states that $\Ct$ itself is residually connected. 

\begin{remark}
By Lemma \ref{lem:strong-flip-descent} a strong quasi-flip is $\emptyset$-homogeneous and each residue admits direct descent regardless of the structure of $\mcC$. As a consequence the local-to-global connectivity and homogeneity result presented here for arbitrary quasi-flips---albeit with some restrictions to $\mcC$---can be considered a variation and generalization of the connectivity result \cite[Proposition 6.6]{Devillers/Muehlherr:2007}. A generalization of the simple connectivity result in \cite{Devillers/Muehlherr:2007} to non-strong quasi-flips, however, seems hard.
\end{remark}

%%%%%%%%%%%%%%%%%%%%%%%%%%%%%%%%%%%%%%%%%%%%%%%%%%%%%%%%%%%%%%%%%%%%%%%%%%%%
%%%%%%%%%%%%%%%%%%%%%%%%%%%%%%%%%%%%%%%%%%%%%%%%%%%%%%%%%%%%%%%%%%%%%%%%%%%%

\section{$\theta$-stability and double coset decompositions} \label{sec:$2$-divisible-rootgrps} \label{sec:double-coset-decomp} \label{sec:stable-twin-apartments}
%  ___         _   _          
% / __| ___ __| |_(_)___ _ _  
% \__ \/ -_) _|  _| / _ \ ' \ 
% |___/\___\__|\__|_\___/_||_|

Let $\mcC$ be a Moufang twin building and let $G$ be a strongly transitive group of automorphisms of $\mcC$. 
The Birkhoff decomposition yields a natural isomorphism of orbit spaces $B_- \backslash G / B_+ \cong W$. Geometrically, this decomposition follows from the fact that (1) the building $G/B_+$ is covered by the positive halves of the twin apartments containing $B_-$ and (2) the group $G$ acts strongly transitively on $\mathcal{C}$.
In this section we give a parametrization of the orbit space $\Gt \backslash G / B_+$ where $\theta$ is a quasi-flip of $\mcC$, based on the observation that under some mild condition each chamber $c$ is contained in a $\theta$-stable twin apartment of $\mcC$. 

\medskip
The condition we impose is that $\theta$ locally fixes opposite chambers in the following sense:

\begin{defn}
\begin{enumerate}
\item An involutive almost isometry $\theta$ of a spherical building $\mcC$ (cf.\ Remark \ref{rmk:spherical-flips}) is said to \Defn{fix opposite chambers}, if each $\theta$-fixed chamber of $\mcC$ admits an opposite $\theta$-fixed chamber.
\item\label{locallyfix} A quasi-flip $\theta$ of a twin building $\mcC$ is said to \Defn{locally fix opposite chambers}, if for each spherical residue $R$ of $\mcC$ that is parallel to $\theta(R)$ and whose type set consists of a single $\theta$-orbit the involutive almost isometry $\mathrm{proj}_R\circ\theta_{|R}$ (cf.\ Remark \ref{localisom}) fixes opposite chambers. \qedhere
\end{enumerate}
\end{defn}

\begin{remark}
Clearly, the residues $R$ in item (\ref{locallyfix}) of the preceding definition have rank one or two.
\end{remark}

The following observation is a variation on a classical theme, cf.\ \cite{Muehlherr:1994}.

\begin{prop} \label{prop:theta-stable-apt-sph}
Let $(G,\Uals,T)$ be a spherical RGD system, let $\mcC$ be the associated spherical building, and let $\theta$ be a quasi-flip of $G$ that locally fixes opposite chambers of $\mcC$. 
If $\theta$ fixes a chamber $c$ of $\mcC$, then there exists a $\theta$-fixed chamber of $\mcC$ opposite $c$.
\end{prop}
\begin{proof}
By Proposition \ref{prop:BN-quasi-flip-Weyl-auto}(\ref{prop:BN-quasi-flip-Weyl-auto2}) the quasi-flip $\theta$ induces an automorphism of $(W,S)$ of order at most $2$. Denote by $\mc{S}$ the set of  $\theta$-orbits in $S$. By hypothesis for each $I\in\mc{S}$ the residue $R_I(c)$ contains a chamber $c_I$ fixed by $\theta$ and opposite $c$ in that residue. Again by hypothesis for each $I' \in \mc{S}$ the residue $R_{I'}(c_I)$ contains a chamber $c_{I,I'}$ fixed by $\theta$ and opposite $c_I$ in that residue.
By \cite[1.32]{Steinberg:1968b} the longest word $w_S$ of $W$ equals a word in the longest words $w_I \in \langle I \rangle$, $I\in\mc{S}$. Therefore the claim follows by induction.
\end{proof}

We now prove that the twin building $\mcC$ can be covered with $\theta$-stable twin apartments.

\begin{prop} \label{thm:theta-stable-apt}
Let $(G,\Uals,T)$ be an RGD system, let $\mcC$ be the associated twin building, and let $\theta$ be a quasi-flip of $G$ that locally fixes opposite chambers of $\mcC$. Then any chamber $c$ is contained in a $\theta$-stable twin apartment.
\end{prop}
\begin{proof}
By Proposition \ref{prop:cox-twisted-inv} and Lemma \ref{lem:pretty-cool-descent} there exist a $\theta$-stable spherical subset $I$ of $S$ and a chamber $d\in\mcC_\eps$ such that $\dt(d)=w_I$ and the convex hull of $d$ and $\theta(d)$ contains $c$ and $\theta(c)$.
By Proposition \ref{new}(\ref{new1}) the residue $R_I(d)$ is a spherical Phan residue. In particular, $R_I(d)$ and $R_I(\theta(d))$ are parallel.
Therefore the product $\mathrm{proj}_{R_I(d)}\circ\theta_{|R_I(d)}$ is an involutive almost isometry of the spherical building $R_I(d)$; cf.\ Remark \ref{localisom}. Since $\mathrm{proj}_{R_I(d)}\circ\theta_{|R_I(d)}$ fixes the chamber $d$, Proposition \ref{prop:theta-stable-apt-sph} implies the existence of a $\mathrm{proj}_{R_I(d)}\circ\theta_{|R_I(d)}$-fixed $d' \in R_I(d)$ opposite $d$ in $R_I(d)$. The chambers $d$, $\theta(d)$, $d'$, $\theta(d')$ are contained in a unique twin apartment, which by construction is $\theta$-stable. Since $c$ and $\theta(c)$ lie in the convex hull of $d$ and $\theta(d)$, the claim follows.
\end{proof}

\begin{lemma}[{Adaption of \cite[Lemma 2.4, Part 2]{Helminck/Wang:1993}}] \label{lem:intersecting-theta-stable-apts-Gt-conj}
Let $\theta$ be a quasi-flip of an RGD system $(G,\Uals,T)$ and let $\Sigma$, $\Sigma'$ be $\theta$-stable twin apartments in the associated twin building with non-empty intersection. Then there exists $g \in \mathrm{Stab}_{\Gt}(\Sigma \cap \Sigma')$ such that $g\Sigma = \Sigma'$.
\end{lemma}
\begin{proof}
Let $c \in \Sigma \cap \Sigma'$. Since $\theta$ stabilizes $\Sigma$ and $\Sigma'$, we have $\theta(c)\in \Sigma\cap\Sigma'$. By \cite[Corollary~8.32]{Abramenko/Brown:2008} the unipotent radical $U$ of the Borel subgroup $B$ that stabilizes $c$ acts sharply transitively on the set of twin apartments containing $c$. Hence there exists $u\in U$ satisfying $\Sigma'=u\Sigma$. As $u$ acts as a building automorphism that fixes $c$, it fixes each chamber in $\Sigma\cap\Sigma'$; in particular it fixes $\theta(c)$. Therefore $\theta(u)$ fixes $c$ by Proposition \ref{prop:BN-quasi-flip-Weyl-auto}(\ref{prop:BN-quasi-flip-Weyl-auto3}), whence $\theta(u)\in U$.
As $\Sigma$, $\Sigma'$ are $\theta$-stable,
$u\Sigma
  = \Sigma'
   = \theta(\Sigma')
   = \theta(u\Sigma)
   = \theta(u) \Sigma$.
Sharp transitivity of the action of $U \ni u,\theta(u)$ on the set of twin apartments containing $c$ (\cite[Corollary~8.32]{Abramenko/Brown:2008}) implies $u=\theta(u)\in \Gt$.
\end{proof}

Inspired by \cite{Rossmann:1979}, \cite{Matsuki:1979}, \cite{Springer:1984}, \cite{Helminck/Wang:1993} we arrive at our desired double coset decomposition.

\begin{proposition}[{Adaption of \cite[Proposition 6.10]{Helminck/Wang:1993}}] \label{prop:DCD-A}
Let $(G,\Uals,T)$ be an RGD system, let $\mcC$ be the associated twin building, let $\theta$ be a quasi-flip of $G$ that locally fixes opposite chambers of $\mcC$, and let $\{ \Sigma_i \mid i \in I\}$ be a set of representatives of the $\Gt$-conjugacy classes of $\theta$-stable twin apartments. If $B$ is a Borel subgroup of $G$, then 
\[\Gt \backslash G / B \isomorph \bigcup_{i \in I} W_{\Gt}(\Sigma_i) \backslash W_{G}(\Sigma_i) , \]
where $W_{X}(\Sigma_i) := \Stab_X(\Sigma_i) / \Fix_X(\Sigma_i)$.
\end{proposition}

\begin{proof}
By Proposition \ref{thm:theta-stable-apt}, to any pair $c$, $c'$ of $\Gt$-conjugate chambers there exist $\Gt$-conjugate $\theta$-stable twin apartments $\Sigma \ni c$, $\Sigma' \ni c'$.
Since by Lemma \ref{lem:intersecting-theta-stable-apts-Gt-conj} intersecting $\theta$-stable apartments are $\Gt$-conjugate, every chamber lies in a unique $\Gt$-orbit of $\theta$-stable apartments, represented by some $\Sigma_i$.
Therefore the orbit space $\Gt \backslash G / B$ can be parametrized by the $\Sigma_i$ and the $\Stab_{\Gt}(\Sigma_i)$-orbits on the chambers in each $\Sigma_i$.
The latter in turn are parametrized by $W_{G}(\Sigma_i)=\Stab_G(\Sigma_i) / \Fix_G(\Sigma_i)$, so that the action of $\Stab_{\Gt}(\Sigma_i)$ on $\Stab_G(\Sigma_i) / \Fix_G(\Sigma_i)$ yields the desired decomposition.
\end{proof}

\cite[Remark 6.11]{Helminck/Wang:1993} yields an alternative parameterization of this double coset space.

\begin{proposition}[{Adaption of \cite[Proposition 6.8]{Helminck/Wang:1993}}]
	\label{prop:DCD-Springer}
Let $(G,\Uals,T)$ be an RGD system, let $\mcC$ be the associated twin building, and let $\theta$ be a quasi-flip of $G$ that locally fixes opposite chambers of $\mcC$.
If $B$ is a Borel subgroup of $G$ stabilizing some chamber $c$ and $\Sigma$ is a $\theta$-stable twin apartment containing $c$, 
then
\[G / B \isomorph \{g\Fix_G(\Sigma) \mid g^{-1}\theta(g)\in \Stab_G(\Sigma) \} . \]
In particular, $$G = \Gt V B,$$ for any set $V$ of $\Gt \times \Fix_G(\Sigma)$-orbit representatives on $\{g \in G \mid g^{-1}\theta(g)\in \Stab_G(\Sigma) \}$.
%
%
%In particular,
%\[\Gt \backslash G / B \isomorph \{\Gt g \Fix_G(\Sigma) \mid g^{-1}\theta(g)\in \Stab_G(\Sigma) \} . \]
\end{proposition}
\begin{proof}
Let $h \in G$ and let $\Sigma'$ be a $\theta$-stable twin apartment containing $h.c$. By strong transitivity, there exists $g\in G$ satisfying $\Sigma'=g.\Sigma$ and $h.c=g.c$. The latter implies $hB = gB$, whence $h$ and $g$ are representing identical cosets. Then $\theta(g.\Sigma)=\theta(g).\Sigma=g.\Sigma$, i.e., $g^{-1}\theta(g)\in \Stab_G(\Sigma)$.
Since $g$ is unique up to right translation by $\Fix_G(\Sigma)$, we obtain $\{g\Fix_G(\Sigma) \mid g^{-1}\theta(g)\in \Stab_G(\Sigma) \} \cong G/B$ via $g\Fix_G(\Sigma) \mapsto gB$. 
\end{proof}

\subsubsection*{Uniquely $2$-divisible root groups}

One key property that ensures that a quasi-flip locally fixes opposite chambers is 
that the root groups be uniquely $2$-divisible, see Propositions \ref{prop:moufang-twodivis-fixpts} and \ref{prop:theta-stable-apt-sph-pol}. This can be considered as a generalization of the concept that the defining characteristic of an algebraic group or a Kac--Moody group be distinct from two; see also Remark \ref{algchar2}.

\begin{defn} \label{2d}
A group $G$ is called \Defn{$2$-divisible}, if for each $g\in G$ there exists $h\in G$ satisfying $h^2=g$. If $h$ is unique with that property, we call $G$ \Defn{uniquely $2$-divisible}.
\end{defn}

\begin{remark} \label{algchar2}
Let $G$ be a connected isotropic reductive linear algebraic group defined over an infinite field $\F$ and let $G(\F)$ be the subgroup of $\F$-rational points of $G$. The root groups of $G(\F)$ are uniquely $2$-divisible if and only if $\charac\F\neq 2$, because they are isomorphic to extensions of $\F$-vector spaces by $\F$-vector spaces; see \cite[proof of Theorem 8.1]{Borel/Tits:1973}, \cite[Definition 3.4]{Steinberg:1973}.
\end{remark}

\begin{proposition} \label{prop:moufang-twodivis-fixpts}
Let $\mouf = (X, (U_x)_{x \in X})$ be a Moufang set. If the root groups $U_x$ are uniquely $2$-divisible, then an involutive automorphism of $\mouf$ fixing a point necessarily fixes a second point.
\end{proposition}
\begin{proof}
Let $\phi$ be an involutive automorphism of $\mouf$ which fixes the point $\infty$. Let $a \in X \backslash \{ \infty \}$. If $\phi(a) = a$, then there is nothing to show, so we may assume $a \neq \phi(a)$. Since $U_\infty$ acts simply transitively on $X\setminus\{\infty\}$, there exists a unique $g\in U_\infty$ such that $g.a = \phi(a)$. By $2$-divisibility there exists an $h\in U_\infty$ such that $h^2=g$. 
We claim that $h.a$ is a fixed point. Indeed, $g.a = \phi(a) = \phi g^{-1}\phi.a.= (g^{-1})^\phi.a$.
Since $U_\infty$ acts simply transitively, we therefore have $g^\phi=g^{-1}$. Since, moreover, $U_\infty$ is uniquely $2$-divisible, this implies $h^\phi=h^{-1}$. Hence
$\phi(h.a) = h^\phi.\phi(a) = h^\phi g.a = h^{-1}h^2.a = h.a$.
\end{proof}

\begin{remark} \label{prop:mouf-2-divis}
Inspection of the explicit list of all Moufang polygons, i.e., Moufang buildings of rank two, presented in \cite[Chapter~16]{Tits/Weiss:2002} implies that for a Moufang set $\mouf = (X, (U_x)_{x \in X})$ that can be embedded into a Moufang polygon the root groups $U_x$ are uniquely $2$-divisible if and only if they are $2$-torsion free. Therefore, whenever it makes sense to assign a characteristic to such a Moufang set, unique $2$-divisibility is equivalent to the defining characteristic being distinct from two.
\end{remark}

\begin{prop} \label{prop:theta-stable-apt-sph-pol}
Let $(G,\Uals,T)$ be an irreducible spherical RGD system of rank two, let $\mcC$ be the associated Moufang polygon, and let $\theta$ be an involutive almost isometry of $\mcC$ which fixes some chamber $c$. 
If the root groups $U_\alpha$ are uniquely $2$-divisible, then there exists a $\theta$-fixed chamber of $\mcC$ opposite $c$.
\end{prop}
\begin{proof}
Let $S = \{ s, t \}$ and let $n$ be the diameter of the Moufang polygon $\mcC$.
If $\theta$ is type-preserving, then the claim holds by the proof of Proposition \ref{prop:theta-stable-apt-sph}.
Therefore we may assume that $\theta(s) = t$.
Let $m:=\lceil \frac{n-1}{2} \rceil$, let $(c=c_0,\ldots,c_{m})$ be a minimal gallery of length $m$, and consider the $\theta$-stable minimal gallery $(\theta(c_{m}),\ldots,c,\ldots,c_{m})$ of length $2m$.
If $n$ is even, then $2m=n$ and the chambers $c_{m}$ and $\theta(c_{m})$ are opposite. Hence $c_{m}$ and $\theta(c_{m})$ are contained in a unique apartment, which for that reason is $\theta$-stable and contains $c$. (See Figure \ref{fig:theta-stable-apt-4gon}.)
\begin{figure}[h]
\centering
\subfloat[Moufang quadrangle]{
  \includegraphics[width=6cm]{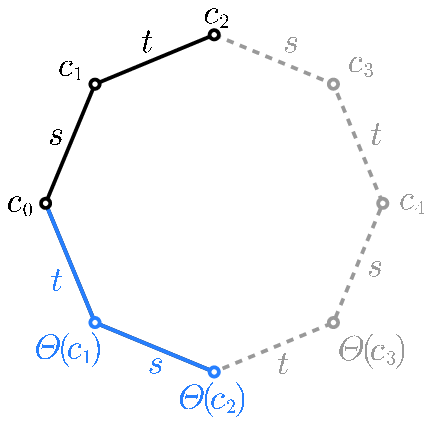}
  \label{fig:theta-stable-apt-4gon}
}
\subfloat[Moufang projective plane]{
  \includegraphics[width=6cm]{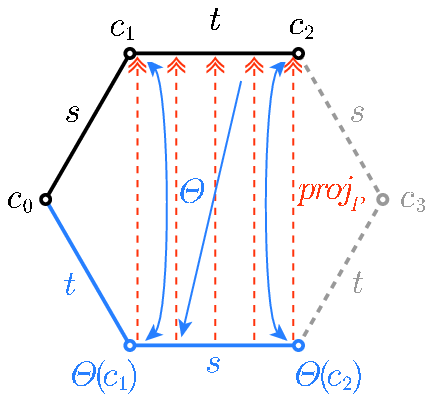}
  \label{fig:theta-stable-apt-3gon}
}
\caption{Constructing a $\theta$-stable apartment inside Moufang polygons.}
\label{fig:theta-stable-apt-sph}
\end{figure}
If $n$ is odd, then there exists $i \in \{ s, t \}$ such that $P_i(c_{m})$ is opposite $\theta(P_i(c_{m}))$. The product $\proj_{P_i(c_{m})} \circ \theta$ is an automorphism of the Moufang set $P_i(c_{m})$, which fixes $c_{m-1}$. Hence by Proposition \ref{prop:moufang-twodivis-fixpts} there is a second chamber $c_{m+1}\in P_i(c_{m})$ fixed by $\proj_{P_i(c_{m})} \circ \theta$, resulting in a $\theta$-stable gallery $(\theta(c_{m+1}),\ldots,c,\ldots,c_{m+1})$ of length $2m+2=n+1$. Hence $c_{m+1}$ is opposite to $\theta(c_{m})$ and the projection of $c_{m+1}$ onto $\theta(P_i(c_{m}))$ equals $\theta(c_{m+1})$. (See Figure \ref{fig:theta-stable-apt-3gon}.) 
Again, we obtain a $\theta$-stable apartment containing $c$. Since $c$ is $\theta$-fixed, necessarily the unique chamber opposite $c$ in that $\theta$-stable apartment has to be $\theta$-fixed as well.
\end{proof}

\subsubsection*{Semi-linear flips}

\begin{defn}
Let $\mathcal{D} = (I,A,\Lambda,(c_i)_{i\in I})$ be a Kac--Moody root datum,
let $\mathcal{F} = (\mathcal{G},(\phi_i)_{i \in I},\eta)$ be the basis of a
Tits functor $\mathcal{G}$ of type $\mathcal{D}$, let $\mathbb{F}$ be a field,
and let $G := \mathcal{G}(\mathbb{F})$ be the corresponding split Kac--Moody
group over $\mathbb{F}$.
A flip $\theta$ of $G$ is called a \Defn{linear} or, for emphasis, an \Defn{$\F$-linear} flip of $G$, if
\begin{itemize}
\item $|\F| \geq 4$ and $\theta$ is an involutive $\F$-automorphism of $G$ (i.e., the field automorphisms $\zeta_i$ in the decomposition of $\theta$ given in \cite[Theorem~4.1]{Caprace:2009} are trivial), or
\item $\F \in \{ \F_2, \F_3 \}$ and $\theta$ is a conjugate of the Chevalley involution.  
\end{itemize}
An \Defn{$\F$-semi-linear} flip of $G$ is an $\F$-linear flip of $G$ composed with a non-trivial field involution of $\F$.
\end{defn}

%\begin{remark}
%Let $\theta$ be an $\F$-semi-linear flip of $G$ and let $\sigma$ be the non-trivial field involution of $\F$ occuring in $\theta$. Then $G_\sigma = \mathrm{Fix}_G(\sigma)$ is a Kac--Moody group over $\mathbb{K} := \mathrm{Fix}_\F(\sigma)$ and $\theta_{|G_\sigma} = (\theta \circ \sigma)_{|G_\sigma}$ is a $\mathbb{K}$-linear flip of $G_\sigma$.
%\end{remark}

\begin{prop}\label{lem:semi-linear-strong}
Let $G$ be a split Kac--Moody group over $\F$ and let $\theta$ be an
$\F$-semi-linear flip of $G$. Then $\theta$ is strong and locally fixes opposite chambers.
\end{prop}
\begin{proof}
Let $\mathcal{C}$ be the twin building corresponding to $G$, let $P$ be a
panel of $\mathcal{C}$ which is parallel to $\theta(P)$, let $G_P \cong
\mathrm{(P)SL_2}(\mathbb{F})$ be the rank one subgroup of $G$ corresponding to
the pair
$P$, $\theta(P)$, and let $\sigma$ be the non-trivial field automorphism of $\mathbb{F}$
involved in $\theta$. The theory of flips of $\mathrm{(P)SL_2}(\mathbb{F})$ as
developed in \cite[Section 4]{Medts/Gramlich/Horn} implies the
existence of a two-dimensional non-degenerate $\sigma$-hermitian sesquilinear form (on the natural module of
$\mathrm{SL_2}(\mathbb{F})$) whose isotropic one-dimensional subspaces are in
one-to-one correspondence to the $\mathrm{proj}_P \circ \theta_{|P}$-fixed
chambers of $P$; cf.\ Remark \ref{localisom}. 
Therefore the fact that $\theta$ locally fixes opposite chambers follows from the
fact that a two-dimensional non-degenerate isotropic $\sigma$-hermitian sesquilinear form with exactly one
isotropic one-dimensional subspace does not exist for $\sigma \neq
\mathrm{id}$. The fact that $\theta$ is strong follows from Lemma
\ref{lem:flip-steep-descent}(\ref{lem:flip-steep-descent2}) and the
observation that a two-dimensional non-degenerate isotropic $\sigma$-hermitian
sesquilinear form without 
anisotropic one-dimensional subspaces does not exist for $\sigma \neq
\mathrm{id}$.
\end{proof}

%%%%%%%%%%%%%%%%%%%%%%%%%%%%%%%%%%%%%%%%%%%%%%%%%%%%%%%%%%%%%%%%%%%%%%%%%%
%%%%%%%%%%%%%%%%%%%%%%%%%%%%%%%%%%%%%%%%%%%%%%%%%%%%%%%%%%%%%%%%%%%%%%%%%%

\section{Applications}

\subsection{Double coset decompositions}

Here we record a double coset decomposition which allows us to simultaneously generalize \cite[Proposition~6.10]{Helminck/Wang:1993} and \cite[Proposition~5.15]{Kac/Wang:1992}.

\begin{theorem} \label{cor:DCD-A-twodivis} \label{cor:DCD-A-twosph}
Let $(G,\Uals,T)$ be an RGD system, let $\mcC$ be the associated twin
building, let $\theta$ be a quasi-flip of $G$, let $B$ be a Borel subgroup of
$G$, let $\Gt$ be the subgroup of $G$ of $\theta$-fixed elements, and let $\{
\Sigma_i \mid i \in I\}$ be a set of representatives of the $\Gt$-conjugacy
classes of $\theta$-stable twin apartments of $\mcC$. If the root groups are
uniquely $2$-divisible or if $G$ is split algebraic/Kac--Moody and $\theta$ is
a semi-linear flip, then
\[\Gt \backslash G / B \isomorph \bigcup_{i \in I} W_{\Gt}(\Sigma_i) \backslash W_{G}(\Sigma_i) , \]
where $W_{X}(\Sigma_i) := \Stab_X(\Sigma_i) / \Fix_X(\Sigma_i)$.
\end{theorem}

\begin{proof}
This is a combination of Propositions \ref{thm:theta-stable-apt} and
\ref{prop:DCD-A} and the discussions on uniquely $2$-divisible root groups and
semi-linear flips in Section \ref{sec:$2$-divisible-rootgrps}.
\end{proof}

\begin{cor} \label{cor:DCD-A-alg}
Let $G$ be a connected isotropic reductive algebraic group defined over an infinite field $\F$, let $\theta$ be an abstract involutive automorphism of $G(\F)$, let $\{A_i\mid i\in I\}$ be a set of representatives of the $\Gt(\F)$-conjugacy classes of $\theta$-stable maximal $\F$-split tori in $G$, and let $P$ a minimal parabolic $\F$-subgroup.
 If the characteristic of $\F$ is distinct from two, then
\[\Gt(\F) \backslash G(\F) / P(\F) \isomorph \bigcup_{i \in I} W_{\Gt(\F)}(A_i) \backslash W_{G(\F)}(A_i) . \]
\end{cor}
\begin{proof}
This follows from Theorem \ref{cor:DCD-A-twosph} plus Remarks \ref{remalg}, \ref{rmk:spherical-flips}, \ref{prop:mouf-2-divis}. 
\end{proof}

\begin{cor} \label{cor:DCD-A-KM}
Let $G$ be a split or quasi-split Kac--Moody group defined over a field $\F$, let $\theta$ be an involutive automorphism of $G$ interchanging the two conjugacy classes of Borel subgroups of $G$, let $\{A_i\mid i\in I\}$ be a set of representatives of the $\Gt$-conjugacy classes of $\theta$-stable $G$-conjugates of the fundamental torus of $G$, and let $B$ be a Borel subgroup of $G$. If the characteristic of $\F$ is distinct from two, then 
\[\Gt \backslash G / B \isomorph \bigcup_{i \in I} W_{\Gt}(A_i) \backslash W_{G}(A_i) . \]
\end{cor}
\begin{proof}
This follows from Theorem \ref{cor:DCD-A-twosph} plus Remark \ref{remkm}.
\end{proof}

Returning once more to a geometric point of view we now generalize the concept of the flip-flop system introduced in Definition \ref{dfn-flipflop}.

\begin{defn}
For $w\in W$ define
$\Ct_w := \{c\in \mcC_+ \mid \dt(c)=w \}$,
where $\dt(c)=\delta^*(c,\theta(c))$ as in \ref{theta-cod}.
Denote the set of all \thcod{}s by
$$\Inv^\theta(\mcC):= \{ w\in W \mid \text{ there exists } c \in \mcC \text{ such that } \dt(c)=w \}. \qedhere $$ 
\end{defn}
Clearly, $\Inv^\theta(\mcC) \subseteq \Inv^\theta(W)$, cf.\ Definition \ref{itheta}.

\begin{defn}
To every quasi-flip $\theta$ of a group $G$, define the \Defn{twist map} $$\tau_\theta:G\to G:g\mapsto \theta(g^{-1})g.$$ The {$\theta$-twisted action} of $G$ on itself is given by $y *_{\theta} g := \theta(g)^{-1}yg$. Its orbits are called the \Defn{$\theta$-twisted $G$-orbits}. Any subgroup $H \leq G$ acts by $\theta$-twisted action on $G$ and gives rise to \Defn{$\theta$-twisted $H$-orbits}.
\end{defn}

\medskip
The following lemma is inspired by the cocompactness proof of \cite[Theorem 7.1]{Gramlich/Witzel}.

\begin{lem} \label{lem:finite-num-Gt-orbits-on-Ct}
Let $\theta$ be a quasi-flip of an RGD system $(G,\Uals,T)$ and let $\mcC$ be the associated twin building.
If each chamber of $\mcC$ is contained in a $\theta$-stable twin apartment, then, for each $w\in \Inv^\theta(\mcC)$, there exists $a\in G$ such that there is a one-to-one correspondence between the $\Gt$-orbits of $\Ct_w$ and the $\theta$-twisted $aTa^{-1}$-orbits of $\tau_\theta(G)\cap aTa^{-1}$.
In particular, if $T$ is finite, there are only finitely many $\Gt$-orbits on every $\Ct_w$.
\end{lem}
\begin{proof}
Let $w\in \Inv^\theta(\mcC)$ and $c\in\Ct_w$. By assumption there exists a $\theta$-stable twin apartment $\Sigma$ containing $c$. The stabilizer $T' := \mathrm{Stab}_G(c,\Sigma)$ is conjugate to $T$, i.e., there exists $a \in G$ with $T'=aTa^{-1}$. 
Strong transitivity of $G$ implies that to any pair $(c',\Sigma')$, where $c'\in\Ct_w$ and $\Sigma' \ni c'$ is a $\theta$-stable twin apartment, there exists $g\in G$ such that $(g.c,g.\Sigma)=(c',\Sigma')$. 
One computes $\delta^*(c',\theta(c')) = w = \delta^*(c,\theta(c)) = \delta^*(g.c,g.\theta(c)) = \delta^*(c',g.\theta(c))$.
Since $\theta(\Sigma) = \Sigma$, one has $g.\theta(c)\in\Sigma'$.
As there is a unique chamber in $\Sigma'$ at any given codistance from $c'$, one concludes that $g.\theta(c)=\theta(c')$. Moreover,
$\tau_\theta(g).c = \theta(g^{-1})g.c = \theta(g^{-1}).c' = \theta(g^{-1}.\theta(c')) = \theta(\theta(c)) = c$
and, similarly, $\tau_\theta(g).\theta(c) = \theta(c)$ and $\tau_\theta(g).\Sigma=\Sigma$. Therefore, $\tau_\theta(g)\in T'$.

By Lemma \ref{lem:intersecting-theta-stable-apts-Gt-conj}, the chambers $c$ and $c'$ are in the same $\Gt$-orbit if and only if $(c,\Sigma)$ and $(c',\Sigma')=g.(c,\Sigma)$ are in the same $\Gt$-orbit. The latter is the case if and only if
there exists $h \in \Gt$ with $h.(c,\Sigma)=g.(c,\Sigma)$. 
As $T' = \mathrm{Stab}_G(c,\Sigma)$, this holds if and only if $h^{-1}g \in T'$ or, equivalently, if and only if there exists $x \in T'$ such that $h=gx$. Since $h \in \Gt$, by an application of $\tau_\theta$ this is equivalent to $1 = \tau_\theta(gx) = \tau_\theta(g) *_{\theta} x$. 

Hence $c$ and $c'$ are in the same $\Gt$-orbit of $\Ct_w$ if and only if $\tau_\theta(g)$ is in the same $\theta$-twisted $aTa^{-1}$-orbit of $\tau_\theta(G)\cap aTa^{-1}$ as $1$. 
\end{proof}

In the proof of \cite[Theorem 7.1]{Gramlich/Witzel}, which deals with the situation that $G$ is an $\F_{q^2}$-locally split Kac--Moody group and $\theta$ semi-linear, Lang's Theorem is applied to conclude that there exists a unique $\theta$-twisted $aTa^{-1}$-orbit in $\tau_\theta(G)\cap aTa^{-1}$, implying transitivity of $\Gt$ on each  $\Ct_w$.
Using this observation we can record once more (after \cite{Gramlich/Muehlherr:unpub-lattices} and \cite{Gramlich/Witzel}) the following special case of Theorem \ref{cor:DCD-A-twodivis}.

\begin{cor}\label{cor}
Let $G$ be a split Kac--Moody group defined over a finite field $\F_{q^2}$, let $\theta$ be an $\F_{q^2}$-semi-linar flip of $G$, and let $B$ be a Borel subgroup of $G$. Then 
\begin{eqnarray*}
\Gt \backslash G / B & \isomorph & \Inv(W), \\
\Gt gB & \leftrightarrow & \dt(gB).
\end{eqnarray*}
\end{cor}

\begin{proof}
As $\theta$ is a semi-linear flip, Propositions  \ref{thm:theta-stable-apt} and \ref{lem:semi-linear-strong}, Lemma \ref{lem:finite-num-Gt-orbits-on-Ct} and the subsequent remark yield an injection from $\Gt \backslash G / B$ into $\Inv(\mcC)$ via $\Gt gB  \mapsto  \dt(gB)$. According to Lemma \ref{lem:flip-steep-descent} it remains to prove that any $\theta$-parallel panel contains a pair of chambers with distinct codistances. The theory of flips of $\mathrm{(P)SL_2}(\mathbb{F}_{q^2})$ as
developed in \cite[Section 4]{Medts/Gramlich/Horn} implies the
existence of a two-dimensional non-degenerate $\sigma$-hermitian sesquilinear form (on the natural module of
$\mathrm{SL_2}(\mathbb{F}_{q^2})$), which necessarily contains both isotropic and anisotropic one-dimensional subspaces. The claim follows.
\end{proof}

\subsection{Flip-flop systems are geometric}

An analysis of involutions of Moufang polygons conducted by Hendrik Van Maldeghem and the second author (see Theorem \ref{hvm}) implies the following consequence of Proposition \ref{prop:hom-inhConn-imply-resConn}, which answers the question whether flip-flop systems of proper flips correspond to a natural diagram geometry.

\begin{theorem} \label{thm:quite-many-flips-are-geometric} \label{thm:geometric-2sph-simply-laced}
Let 
\begin{itemize}
\item $(G,\Uals,T)$ be an RGD system of $2$-spherical type with finite root groups $\Uals$ of odd order and of cardinality at least five, or 
\item let $\mathbb{F}$ be an infinite field of characteristic distinct from two and let $G$ be a connected $\F$-split reductive
$\F$-group or a split Kac--Moody
group over $\F$ of $2$-spherical type $(W,S)$.
\end{itemize}
 Moreover, let $\mcC$ be the associated twin building, and let $\theta$ be a quasi-flip of $G$.

 Then the flip-flop system $\Ct$ is connected and equals the union of the minimal Phan residues of $\mcC_\eps$. Moreover, there exists $K \subseteq S$, such that $(\mcC, \theta)$ is $K$-homogeneous. Furthermore, the $K$-residue chamber system $\Ct_K$ is connected and residually connected. In particular, if $\theta$ is proper, then $(\mcC, \theta)$ is $\emptyset$-homogeneous, and $\Ct$ is residually connected.
\end{theorem}

\begin{proof}
Combine Theorem \ref{hvm} and Proposition
\ref{prop:hom-inhConn-imply-resConn} with Remarks \ref{remalg}, \ref{remkm} and \ref{prop:mouf-2-divis}.
\end{proof}

\begin{remark}
This largely answers the question from
\cite{Bennett/Gramlich/Hoffman/Shpectorov:2003} whether a flip-flop system of a proper flip is geometric in general, as residual connectedness implies this by \cite[2.2--2.4]{Tits:1981}, \cite[Chapter 3]{Buekenhout/Cohen}.
\end{remark}

Let $G$ be a connected $\F$-split reductive $\F$-group and let $\theta$ be an involutive $\F$-automorphism of $G(\F)$.
By \cite[Propositions 4.8 and 4.11]{Helminck/Wang:1993} all minimal $\theta$-split parabolic $\F$-subgroups have equal type.
Theorem \ref{thm:quite-many-flips-are-geometric} implies the following generalization.

\begin{cor}
Let $G$ be a connected $\F$-split reductive $\F$-group defined over an infinite field $\F$ of characteristic distinct from two and let $\theta$ be an abstract involutive automorphism of $G(\F)$. Then the minimal $\theta$-split parabolic $\F$-subgroups of $G(\F)$ all have equal type.
\end{cor}
\begin{proof}
Minimal $\theta$-split parabolic $\F$-subgroups correspond one-to-one to
minimal Phan residues of that same type. In view of Remark
\ref{abstractpossible}(\ref{abstractpossible1}) Theorem \ref{thm:quite-many-flips-are-geometric} applies.
\end{proof}

\subsection{$\Gt$ is finitely generated over finite fields}

Let $\mcC=((\mcC_+,\delta_+),(\mcC_-,\delta_-),\delta^*)$ be a locally finite Moufang twin building of rank $n$, i.e., each chamber of $\mcC$ is adjacent to finitely many chambers, and let $q \in \mathbb{N}$ be minimal with the property that $\mcC$ contains a panel of order $q+1$.
If $q > n$, then $G_\theta$ is a lattice in the full automorphism group $\mathrm{Isom}(\mcC_+)$ endowed with the topology of compact convergence, as announced by the first and the third author in 2007 and published for semi-linear flips $\theta$ in \cite{Gramlich/Muehlherr:unpub-lattices}. (Note that Lemma \ref{lem:pretty-cool-descent}, Proposition \ref{thm:theta-stable-apt} and Lemma \ref{lem:finite-num-Gt-orbits-on-Ct} of the present article combined with the strategy of \cite{Gramlich/Muehlherr:unpub-lattices} provide a proof for arbitrary quasi-flips $\theta$.)

If $q > \frac{1764^n}{25}$ and if $\mcC$ is $2$-spherical, then the locally compact group $\mathrm{Isom}(\mcC_+)$ satisfies Kazhdan's property (T) by \cite{Dymara/Januszkiewicz:2002}. Hence the lattice $\Gt$ also satisfies property (T) by \cite[Theorem~1.7.1]{Bekka/Harpe/Valette:2008} and, thus, is finitely generated by \cite[Theorem~1.3.1]{Bekka/Harpe/Valette:2008}.
The following result deals with the finite generation of $\Gt$ without using Kazhdan's property (T).

\begin{theorem} \label{thm:loc-fin-KM-is-fin-gen}
 Let 
$(G,\Uals,T)$ be an RGD system of $2$-spherical type with finite root groups $\Uals$ of odd order and of cardinality at least five, and let $\theta$ be a quasi-flip of $G$.
Then the group $\Gt$ is finitely generated.
\end{theorem}
\begin{proof}
Since the torus $T$ of $G$ is a finite group and since the $U_\alpha$ are uniquely $2$-divisible, Proposition \ref{thm:theta-stable-apt} allows us to apply Lemma \ref{lem:finite-num-Gt-orbits-on-Ct} to obtain a finite set $X$ of representatives of $\Gt$-orbits on the flip-flop system $\Ct$ of the twin building $\mcC$ associated to $G$. Let $c \in \mcC$, $m := |X|$, and $C := B_{2m}(c) \cap \Gt.c$. For each $d \in C$, let $g_d \in \Gt$ such that $g_d(c) = d$. Note that $C$ is a finite set, as $\mcC$ is locally finite. Since $\Ct$ is connected by Theorem \ref{thm:quite-many-flips-are-geometric}, the pigeon hole principle implies $$\Ct \subseteq \bigcup_{d \in \Gt.c} B_{m}(d).$$ Therefore $$\Gt=\gen{\Stab_{\Gt}(c) \cup \{ g_d \mid d \in C \} }.$$
Since the stabilizer $\Stab_{\Gt}(c)$ is finite by \cite[Corollary 3.8]{Caprace/Muehlherr:2006} and since $C$ is finite, the group $\Gt$ is finitely generated.
\end{proof}

\begin{example}
Consider the quasi-flip $\theta$ of the group $\mathrm{GL}_n(\mathbb{F}_q[t,t^{-1}])$ described in Example \ref{Iwahori}, i.e., $\theta$ fixes $\mathrm{GL}_n(\mathbb{F}_q)$ elementwise and interchanges $t$ and $t^{-1}$. Then $G_\theta = \mathrm{GL}_n(\mathbb{F}_q([t+t^{-1}])) \cong \mathrm{GL}_n(\mathbb{F}_q([X]))$. If $n \geq 3$ (in which case the group $\mathrm{GL}_n(\mathbb{F}_q[t,t^{-1}])$ is $2$-spherical), then it has already been proved in \cite[Hauptsatz~C(iii)]{Behr:1969} that the group $G_\theta \cong \mathrm{GL}_n(\mathbb{F}_q([X]))$ is finitely generated. On the other hand, by \cite{Nagao:1959} the group $G_\theta \cong \mathrm{GL}_2(\mathbb{F}_q([X]))$ is not finitely generated; note that the group $\mathrm{GL}_2(\mathbb{F}_q[t,t^{-1}])$ is of type $\tilde A_1$, i.e., it is not $2$-spherical. 
\end{example}

\begin{remark}
More generally, if $G$ is a non-$2$-spherical Kac--Moody group over $\F_{q}$, then there exist quasi-flips $\theta$ such that $G_\theta$ is not finitely generated, as observed by Caprace, the first and the third author. Indeed, let $T$ be a tree residue of type $\{s, t \}$ of the building, so that $G.T$ is a simplicial tree by \cite[Proposition 2.1]{Dymara/Januszkiewicz:2002}. The action of the lattice $\Gt$ on the simplicial tree $G.T$ is minimal. If $\theta$ is such that the cardinality of $\{ \dt(c) \in W \mid c \in T \}$ is infinite (which by Corollary~\ref{cor} for instance is the case if $\theta$ is a semi-linear flip), then there are infinitely many $\Gt$-orbits on $G.T$. Hence from \cite[Proposition 7.9]{Bass:1993} (also \cite[Proposition 5.6]{Bass/Lubotzky:2001}) it follows that $\Gt$ cannot be finitely generated.
\end{remark}

{\footnotesize

\bibliography{inv}{}
\bibliographystyle{alpha}

}

\vspace{1cm}
\noindent
Authors' addresses:

\vspace{.8cm}

\noindent 
Ralf Gramlich \\
TU Darmstadt \\
Fachbereich Mathematik \\
Schlo\ss gartenstra\ss e 7 \\
64289 Darmstadt \\
Germany \\
e-mail: {\tt gramlich@mathematik.tu-darmstadt.de}

\vspace{.4cm}

\noindent Alternative address:

\noindent University of Birmingham \\
School of Mathematics \\
Edgbaston \\
Birmingham \\
B15 2TT \\
United Kingdom \\
e-mail: {\tt ralfg@maths.bham.ac.uk}

\vspace{.8cm}

\noindent Max Horn \\
TU Braunschweig \\
Fachbereich Mathematik und Informatik \\
Institut Computational Mathematics \\
Pockelstra\ss e 14 \\
38106 Braunschweig \\
Germany \\
email: {\tt max.horn@tu-bs.de}

\vspace{.8cm}

\noindent Bernhard M\"uhlherr \\
Universit\"at Gie\ss en \\
Mathematisches Institut \\
Arndtstra\ss e 2 \\
35392 Gie\ss en \\
Germany \\
email: {\tt Bernhard.M.Muehlherr@math.uni-giessen.de}

\end{document}